\documentclass{tac}
\usepackage{mathpartir,cancel,cmll,textcomp,environ}
\usepackage{amssymb,amsmath,stmaryrd,mathrsfs}
\usepackage[all,2cell]{xy}
\usepackage{tikz,tikz-cd}
\usetikzlibrary{arrows}
\usepackage{enumitem}
\usepackage{xcolor}
\usepackage{mathtools}          
\usepackage{microtype}
\usepackage{url}                
\usepackage{xspace}             
\usepackage{mathpartir}
\usepackage{array}
\usepackage{subcaption}

\makeatletter
\let\ea\expandafter

\newcount\foreachcount

\def\foreachletter#1#2#3{\foreachcount=#1
  \ea\loop\ea\ea\ea#3\@alph\foreachcount
  \advance\foreachcount by 1
  \ifnum\foreachcount<#2\repeat}

\def\foreachLetter#1#2#3{\foreachcount=#1
  \ea\loop\ea\ea\ea#3\@Alph\foreachcount
  \advance\foreachcount by 1
  \ifnum\foreachcount<#2\repeat}

\def\definecal#1{\ea\gdef\csname c#1\endcsname{\ensuremath{\mathcal{#1}}\xspace}}
\foreachLetter{1}{27}{\definecal}
\def\definefrak#1{\ea\gdef\csname f#1\endcsname{\ensuremath{\mathfrak{#1}}\xspace}}
\foreachletter{1}{9}{\definefrak}
\foreachletter{10}{27}{\definefrak}
\foreachLetter{1}{27}{\definefrak}

\newcommand{\nCComon}{\mathrm{CComon}}
\newcommand{\nShuf}{\mathrm{Shuf}}

\newcommand{\op}{^{\mathrm{op}}}
\let\ten\otimes
\DeclareMathOperator\ob{ob}
\let\toto\rightrightarrows
\let\xto\xrightarrow
\def\toiso{\xto{\smash{\raisebox{-.5mm}{$\scriptstyle\sim$}}}}

\let\your@state\state
\def\state#1{\my@state#1}
\def\my@state#1.{\gdef\currthmtype{#1}\your@state{#1.}}
\let\your@staterm\staterm
\def\staterm#1{\my@staterm#1}
\def\my@staterm#1.{\gdef\currthmtype{#1}\your@staterm{#1.}}
\let\your@section\section
\def\section{\gdef\currthmtype{section}\your@section}
\let\your@subsection\subsection
\def\subsection{\gdef\currthmtype{subsection}\your@subsection}
\let\your@figure\figure
\def\figure{\gdef\currthmtype{Figure}\your@figure}
\def\currthmtype{}
\def\cref#1{\ref*{label@name@#1}~\ref{#1}}
\AtBeginDocument{%
  \let\old@label\label%
  \def\label#1{%
    {\let\your@currentlabel\@currentlabel%
      \edef\@currentlabel{\currthmtype}%
      \old@label{label@name@#1}}%
    \old@label{#1}}
}

\newtheorem{thm}{Theorem}
\let\your@thm\thm
\def\thm{\your@thm\gdef\currthmtype{Theorem}}
\newtheorem{prop}{Proposition}
\let\your@prop\prop
\def\prop{\your@prop\gdef\currthmtype{Proposition}}
\newtheorem{lem}{Lemma}
\let\your@lem\lem
\def\lem{\your@lem\gdef\currthmtype{Lemma}}
\newtheoremrm{defn}{Definition}
\let\your@defn\defn
\def\defn{\your@defn\gdef\currthmtype{Definition}}
\newtheoremrm{rmk}{Remark}
\let\your@rmk\rmk
\def\rmk{\your@rmk\gdef\currthmtype{Remark}}

\def\pr@@f[#1]{\subsubsection*{\sc #1.}}
\let\qed\endproof

\setitemize[1]{leftmargin=2em}
\setenumerate[1]{leftmargin=*}

\let\c@equation\c@subsection
\numberwithin{equation}{section}

\let\ep\varepsilon
\let\si\sigma

\makeatother

\tikzset{ed/.style={auto,inner sep=0pt,font=\scriptsize}} 
\tikzset{>=stealth'}
\tikzset{vert/.style={draw,circle,inner sep=1pt,fill=white}}
\tikzset{vert2/.style={draw,circle,inner sep=2pt,fill=white}}
\def\rdual#1{{#1}^*}
\def\tr{\mathrm{tr}}
\def\idfunc{\mathsf{id}}
\def\id{\mathsf{id}}
\def\eval#1#2{#1 \mathbin{\triangleleft} #2}
\def\abs#1#2{\lambda^{#1} #2}
\let\comult\triangle

\def\inv#1{\overline{#1}}
\def\invt#1{\widehat{#1}}
\def\ec{()} 
\let\types\vdash
\newcommand{\preterm}{\;\mathsf{term}}
\def\s#1{_{\scriptscriptstyle(#1)}}
\def\ss#1#2{_{\scriptscriptstyle(#1)(#2)}}
\newcommand{\actv}[1]{{#1}^\star}
\newcommand{\atleastone}[1]{{#1}_{\scriptscriptstyle \ge 1}}

\newcommand{\oneactv}[1]{{#1}^{\star\scriptscriptstyle\ge 1}}
\newcommand{\zeroactv}[1]{{#1}^{\star\scriptscriptstyle = 0}}
\newcommand{\allactv}[1]{{#1}^{\star}}
\newcommand{\maybeactv}[1]{{#1}^{\star\scriptscriptstyle?}}
\newcommand{\F}[1]{\mathfrak{F}_{#1}}
\def\pair#1#2{\langle #1,#2\rangle}
\def\defeq{\mathrel{\smash{\overset{\scriptscriptstyle\mathsf{def}}{=}}}}
\def\hy{\mathrm{hy}}

\title{A practical type theory for symmetric monoidal categories}
\author{Michael Shulman}\thanks{This material is based
    on research sponsored by The United States Air Force Research
    Laboratory under agreement number FA9550-15-1-0053.  The
    U.S. Government is authorized to reproduce and distribute reprints
    for Governmental purposes notwithstanding any copyright notation
    thereon.  The views and conclusions contained herein are those of
    the author and should not be interpreted as necessarily
    representing the official policies or endorsements, either
    expressed or implied, of the United States Air Force Research
    Laboratory, the U.S. Government, or Carnegie Mellon University.}
\eaddress{shulman@sandiego.edu}
\address{Department of Mathematics\\University of San Diego\\
  5998 Alcala Park\\San Diego, CA, 92110\\USA}

\begin{document}

\maketitle

\begin{abstract}
  We give a natural-deduction-style type theory for symmetric monoidal
  categories whose judgmental structure directly represents morphisms
  with tensor products in their codomain as well as their domain.  The
  syntax is inspired by Sweedler notation for coalgebras, with
  variables associated to types in the domain and terms associated to
  types in the codomain, allowing types to be treated informally like
  ``sets with elements'' subject to global linearity-like
  restrictions.  We illustrate the usefulness of this type theory with
  various applications to duality, traces, Frobenius monoids, and
  (weak) Hopf monoids.
\end{abstract}

\setcounter{tocdepth}{1}
\tableofcontents

\section{Introduction}
\label{sec:introduction}

\subsection{Type theories for monoidal categories}
\label{sec:ttmon}

Type theories are a powerful tool for reasoning about categorical
structures.  This is best-known in the case of the internal language
of a topos, which is a higher-order intuitionistic logic.  But there
are also weaker type theories that correspond to less
highly-structured categories, such as regular logic for regular
categories, simply typed $\lambda$-calculus for cartesian closed
categories, typed algebraic theories for categories with finite
products, and so on (a good overview can be found in~\cite[Part
D]{ptj:elephant}).

However, type theories seem to be only rarely used to reason about
\emph{non-cartesian monoidal} categories.  Such categories are of
course highly important in both pure mathematics and its applications
(such as quantum mechanics, network theory, etc.), but usually they
are studied using either traditional arrow-theoretic syntax or string
diagram calculus~\cite{js:geom-tenscalc-i,selinger:graphical}.

Type theories corresponding to non-cartesian monoidal categories do
certainly exist --- intuitionistic linear logic for closed symmetric
monoidal categories, ordered logic for non-symmetric monoidal
categories, classical linear logic for $\ast$-autonomous categories
--- but they are not widely used for reasoning about monoidal
categories.  I believe this is largely because their convenience for
practical category-theoretic arguments does not approach that of
cartesian type theories.  The basic problem is that most nontrivial
arguments in monoidal category theory involve tensor product objects
in both the \emph{domain} and the \emph{codomain} of morphisms.

Existing type theories for non-cartesian monoidal categories fall into
two groups, neither of which deals satisfactorily with this issue.
The first is exemplified by intuitionistic linear logic, whose
judgments have many variables as inputs but only a single term as
output:
\[ x:A, y:B, z:C \types f(x,g(y,z)): D .
\]
This allows the terms (such as $f(x,g(y,z))$) to intuitively ``treat
types as if they were sets with elements'' (one of the big advantages
of type-theoretic syntax).  But it privileges tensor products in the
domain (here the domain is semantically $A\otimes B\otimes C$) over
the codomain.  The only way to talk about morphisms into tensor
products is to use the tensor product type constructor:
\[ x:A, y:B \types (y,x): B\otimes A.
\]
This is asymmetric, but more importantly its term syntax is heavy and
unintuitive: in the cartesian case we can use an element $p:A\times B$
by simply projecting out its components $\pi_1(p):A$ and $\pi_2(p):B$,
but in the non-cartesian case we have to ``match'' it and bind both
components as new variables:
\[ p:A\otimes B \types \mathsf{let}\; (x,y)
  \coloneqq p \;\mathsf{in}\; (y,x) : B\otimes A.
\]

The second group of non-cartesian type theories is exemplified by
classical linear logic, whose judgments have multiple inputs as well
as multiple outputs:
\[ A,B,C \types D,E.
\]
This eliminates the asymmetry, but at the expense of no longer having
a concise and intuitive term syntax.  Most commonly such type theories
are presented as sequent calculi without any terms at all.  Term
calculi do exist
(e.g.~\cite{reddy:dirprolog,bcst:natded-coh-wkdistrib,ong:sem-classical,cpt:ll-session})
but usually involve some kind of ``covariables'', which complicate the
interpretation of types as ``sets with elements''.  Moreover, in type
theories of this sort the tensor product appearing in codomains is
usually a \emph{different} one from the one appearing in domains (the
``cotensor product'' of a $\ast$-autonomous category), whereas in
practical applications it happens very frequently that they are the
same.  The only type theories with the same tensor product appearing
in domains and codomains that I know of in the literature
are~\cite{shirahata:seqcalc-cptclosed,duncan:types-qcomp}, which are
sequent calculi without a real term syntax,
and~\cite{hasegawa:thesis}, which does have a term syntax but still
decomposes tensor products by matching.  (The referee has pointed out
that~\cite{co:flang-cyclic} is in a similar spirit to our work, though
not directly comparable.)

The purpose of this paper is to fill this gap, by describing a type
theory for symmetric monoidal categories in which:
\begin{itemize}
\item Judgments have multiple inputs as well as multiple outputs.
\item The tensor products appearing in domains and codomains are the same.
\item There is a convenient and intuitive term syntax without the need
  for ``covariables''.
\item Tensor product types, though rarely needed, have a convenient
  term syntax involving ``projections'', as in the cartesian case,
  rather than ``matches''.
\end{itemize}
This type theory seems to be very convenient in practice for reasoning
about symmetric monoidal categories.  We will show this in
\cref{sec:examples} with a number of examples; here we sketch two of
them to whet the reader's appetite.

For the first, recall that in a compact closed category, the
\textbf{trace} $\tr(f)$ of an endomorphism $f:A\to A$ is the composite
\[ I \xto{\eta} A \otimes A^* \xto{f\otimes \id}
  A\otimes A^* \toiso A^* \otimes A \xto{\ep} I
\]
where $I$ is the unit object, $A^*$ is the dual of $A$, and $\eta$ and
$\ep$ are the unit and counit of the duality (also called the
coevaluation and evaluation, respectively).  A fundamental property of
trace is \emph{cyclicity}, i.e.\ $\tr(f g) = \tr(g f)$ for any
$f:A\to B$ and $g:B\to A$.  The proof of this is essentially
incomprehensible when written with composites of morphisms:
\begin{multline*}
  \scriptstyle\tr(fg) = \ep \fs (fg \ten \id) \eta = \ep\fs(f\ten
  \id)(g\ten \id)\eta =
  \ep\fs(f\ten \id)(\id\ten\ep\ten\id)(\eta\ten\id\ten\id)(g\ten \id)\eta =\\
  \scriptstyle\ep\fs(\id\ten\ep\ten\id)(f\ten
  \id\ten\id\ten\id)(\id\ten\id\ten g\ten \id)(\eta\ten \eta) =
  (\ep\ten\ep)\fs(\id\ten\id\ten g\ten \id)(f\ten \id\ten\id\ten\id)(\eta\ten \eta) =\\
  \scriptstyle(\ep\ten\ep)\fs(\id\ten\id\ten g\ten
  \id)(\id\ten\id\ten \eta)(f\ten \id)\eta =
  \ep\fs(\id\ten\ep\ten\id)(g\ten\id\ten\id\ten\id)(\eta\ten\id\ten\id)(f\ten\id)\eta =\\
  \scriptstyle\ep\fs(g\ten\id)(\id\ten\ep\ten\id)(\eta\ten\id\ten\id)(f\ten\id)\eta
  = \ep\fs(g\ten\id)(f\ten\id)\eta = \ep\fs(gf\ten\id)\eta = \tr(gf).
\end{multline*}
It becomes much more understandable when written in string diagram
calculus, as in
\cref{fig:cyclicity}.
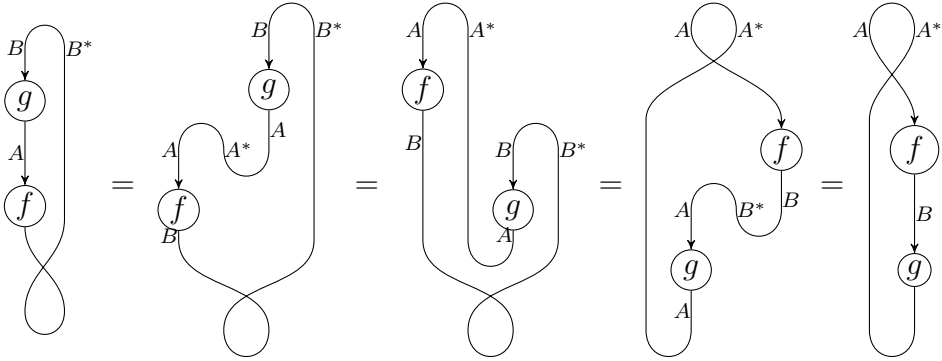
\begin{figure}
  \begin{tabular}{m{10mm}m{2mm}m{22mm}m{2mm}m{22mm}m{2mm}m{19mm}m{0mm}m{10mm}}
  \begin{tikzpicture}[scale=.7]
   \node[vert2](f) at (0, -1){$g$};
   \node[vert](g) at (0, -3){$f$};
   
   \draw[->](f)--node[ed,swap]{$A$}(g);

   \draw[->](g)to [out=-90, in =90] (.75, -5)to [out=-90, in =-90, looseness=2] (-.0,-5)
      to [out=90, in =-90, looseness=1] 
      (.75,-3.25)--node[ed,swap, at end]{$\rdual{B}$}(.75, 0)
      to [out=90, in =90, looseness=2] (0, 0)--
     node[ed,swap, at start]{$B$}(f);


\end{tikzpicture}
&=&

  \begin{tikzpicture}[scale=.8]
   \node[vert2](f) at (0, -1){$g$};
   \node[vert](g) at (-1.5, -3){$f$};
   
   \draw[->](f)--node[ed]{$A$}(0, -2) to [out=-90, in =-90, looseness=2]
   (-.75, -2) node[ed, right]{$\rdual{A}$}
     to [out=90, in =90, looseness=2](-1.5,-2)--node[ed, swap, at start]{$A$} (g);

     \draw[->](g)--node[ed, swap]{$B$}(-1.5, -3.5) to [out=-90, in =90]
     (0, -5)to [out=-90, in =-90, looseness=2] (-.75,-5)
      to [out=90, in =-90, looseness=1]
      (.75,-3.5)--node[ed, swap, at end]{$\rdual{B}$}(.75, 0)
      to [out=90, in =90, looseness=2] (0, 0)--
     node[ed, at start, swap]{$B$}(f);

\end{tikzpicture}
&=&

\begin{tikzpicture}[scale=.8]
   \node[vert2](f) at (0, -3){$g$};
   \node[vert](g) at (-1.5, -1){$f$};
   
   \draw[->](f) --node[ed, swap]{$A$} (0, -3.5)
   to [out=-90, in =-90, looseness=2](-.75, -3.5) --
   node[ed, right, at end]{$\rdual{A}$}(-.75, 0)
   to [out=90, in =90, looseness=2](-1.5,0)
   --node[ed, swap, at start]{$A$}  (g);

   \draw[->](g)--node[ed, swap, near start]{$B$}(-1.5, -3.5)
   to [out=-90, in =90](0, -5)to [out=-90, in =-90, looseness=2] (-.75,-5)
   to [out=90, in =-90, looseness=1](.75,-3.5)
   --node[ed, swap, at end]{$\rdual{B}$}(.75, -2) to
   [out=90, in =90, looseness=2] (0, -2)--
            node[ed, at start, swap]{$B$}(f);

\end{tikzpicture} 
&=&

  \begin{tikzpicture}[scale=.8]
   \node[vert2](f) at (-3, -4){$g$};
   \node[vert](g) at (-1.5, -2){$f$};
   
   \draw[->](f) --node[ed, swap]{$A$}(-3, -5) to
   [out=-90, in =-90, looseness=2](-3.75, -5) --
     (-3.75, -1.5) to[out=90, in =-90, looseness=.8]
     node[ed, right, at end]{$\rdual{A}$} (-2.25, 0)
     to [out=90, in =90, looseness=2](-3,0)to
     [out=-90, in =90, looseness=.8] 
node[ed, swap, at start]{$A$}(-1.5, -1.5)-- (g);

\draw [->](g)--node[ed, near end]{$B$}(-1.5, -3)to
[out=-90, in =-90, looseness=2] (-2.25,-3)
to [out=90, in =90]node[ed, swap, at end]{$\rdual{B}$}
(-2.25, -3) to [out=90, in =90, looseness=2] (-3, -3)--
       node[ed, at start, swap]{$A$}(f);

\end{tikzpicture}
&=&
  \begin{tikzpicture}[scale=.8]
   \node[vert](f) at (-1.5, -4){$g$};
   \node[vert2](g) at (-1.5, -2){$f$};
   
   \draw[->](f) --(-1.5, -5) to [out=-90, in =-90, looseness=2](-2.25, -5) --
     (-2.25, -1.5) to[out=90, in =-90, looseness=.8]
     node[ed, right, at end]{$\rdual{A}$} (-1.5, 0)
     to [out=90, in =90, looseness=2](-2.25,0)to[out=-90, in =90, looseness=.8] 
node[ed, swap, at start]{$A$}(-1.5, -1.5)-- (g);

   \draw [->](g)--node[ed]{$B$}(f);

\end{tikzpicture}
\end{tabular}
\caption{Cyclicity of trace with string diagrams}
  \label{fig:cyclicity}
\end{figure}
But in our type theory, the proof is one line of algebra:
\begin{equation*}
\tr(f g) \defeq (\,\mid \eval{\abs{B} y}{fgy})
  = (\,\mid \eval{\abs{B}{y}}{fx},\eval{\abs{A}{x}}{gy})
  = (\,\mid \eval{\abs{A} x}{gfx})
  \defeq \tr(g f).
\end{equation*}
The reader is not expected to understand this proof yet, but to give
some flavor we explain that the two non-definitional equalities are
``$\beta$-reductions for duality''.  The binary operator $\eval{}{}$
denotes the counit $\ep$, while a variable/abstraction pair
$(x,\abs{A}{x})$ denotes the unit $\eta$.  A term of the form
$\eval{\abs{A}{x}}{t}$, where $x$ does not occur in $t$, is a
$\beta$-redex for duality, and in the resulting reduction this term is
eliminated and the occurrence of $x$ elsewhere\footnote{This
  ``elsewhere'' could be anywhere at all.  Thus although $\abs{A}{x}$
  does ``bind'' the variable $x$, it does not delimit its scope.} is
substituted by $t$.  (Semantically, this means applying a zigzag
identity, which in string diagram notation is ``straightening a bent
string''.)  Thus the middle term
$(\,\mid \eval{\abs{B}{y}}{fx},\eval{\abs{A}{x}}{gy})$ contains two
$\beta$-redexes, and the resulting two reductions yield the two terms
on either side.

For our second example, recall that if $x$ is an element of a monoid
$M$, and $\inv{x}$ and $\invt{x}$ are both two-sided
inverses\footnote{In fact, of course, it suffices for $\inv{x}$ to be
  a left inverse and $\invt{x}$ a right inverse.} of $x$, then
$\inv{x}=\invt{x}$ by the following calculation:
\begin{equation}
  \inv{x} = \inv{x}e = \inv{x}(x\invt{x}) = (\inv{x}x)\invt{x} = e\invt{x} = \invt{x}.\label{eq:intro-classical-inversion}
\end{equation}
This same proof can be reproduced in cartesian type theory, and
therefore holds for monoid objects in any cartesian monoidal category.
It follows that if a monoid object $M$ has an ``inversion'' morphism
$i:M\to M$ such that
\begin{equation}\label{eq:intro-inversion}
  \vcenter{\xymatrix@C=1pc{& M\times M \ar[rr]^{i\times 1} && M\times M \ar[dr]^m \\
      M \ar[ur]^{\comult} \ar[dr]_{\comult} \ar[rr]^{!} && 1 \ar[rr]^{e} && M \\
      & M\times M \ar[rr]_{1\times i} && M\times M \ar[ur]_{m}}}
\end{equation}
commutes (thereby making it a group object), then $i$ is the unique
such morphism.

An interesting application of this relies on the fact that the
category $\nCComon(\cC)$ of \emph{cocommutative comonoids} in any
symmetric monoidal category \cC is in fact \emph{cartesian} monoidal,
with the cartesian product being the monoidal structure of \cC.  A
monoid object in $\nCComon(\cC)$ is a \emph{cocommutative bimonoid
  object} in \cC: an object with both a monoid and a cocommutative
comonoid structure, such that the monoid structure maps are comonoid
morphisms (or equivalently the comonoid structure maps are monoid
morphisms).  For any bimonoid object, a morphism $i:M\to M$ satisfying
the analogue of~\eqref{eq:intro-inversion}:
\begin{equation}\label{eq:intro-antipode}
  \vcenter{\xymatrix@C=1pc{& M\otimes M \ar[rr]^{i\otimes 1} && M\otimes M \ar[dr]^m \\
      M \ar[ur]^{\comult} \ar[dr]_{\comult} \ar[rr]^{\varepsilon} && I \ar[rr]^{e} && M \\
      & M\otimes M \ar[rr]_{1\otimes i} && M\otimes M \ar[ur]_{m}}}
\end{equation}
(where $\comult,\varepsilon$ are the comonoid structure) is called an
\emph{antipode}, and a bimonoid object equipped with an antipode is
called a \emph{Hopf monoid}.  Thus, the classical argument for
uniqueness of inverses, phrased in cartesian type theory, implies that
a cocommutative bimonoid object has at most one antipode.  Dualizing,
we can conclude that a \emph{commutative} bimonoid object also has at
most one antipode.

For bimonoids that are neither commutative nor cocommutative,
cartesian type theory cannot help us.  We can manually translate the
classical proof~\eqref{eq:intro-classical-inversion} into commutative
diagrams and then replace all the $\times$'s by $\otimes$'s, but this
is tedious and one feels that there should be a better way.  And
indeed, in our type theory for symmetric monoidal categories, we can
write a very close analogue of~\eqref{eq:intro-classical-inversion}
that applies to arbitrary bimonoids:
\begin{equation}
  \inv{x} = \inv{x}\,e = (\inv{x\s1}\, e \mid \cancel{x\s2}) =
  \inv{x\s1}\, x\s2 \,\invt{x\s3} = (e\, \invt{x\s2} \mid \cancel{x\s1})
  = e\,\invt{x} = \invt{x}.
  \label{eq:intro-antipode-uniq}
\end{equation}
Again, the reader is not expected to understand this completely, but
the resemblance to~\eqref{eq:intro-classical-inversion} should be
clear.  The main differences are the subscripts on the $x$'s and the
insertion of the steps $(\inv{x\s1}\, e \mid \cancel{x\s2})$ and
$(e\, \invt{x\s2} \mid \cancel{x\s1})$.  The subscripts are used to
track applications of the comultiplication $\comult:M\to M\otimes M$.
This is necessary because (unlike in the cartesian case) not all
morphisms are comonoid morphisms (so that $f(x)\s1$ might differ from
$f(x\s1)$), and also because in the absence of cocommutativity the
resulting ``copies'' of $x$ must be distinguished (so that
$(x\s1,x\s2)$ is different from $(x\s2,x\s1)$).  Similarly, the
canceled terms such as $\cancel{x\s2}$ track applications of the
counit $\ep : M\to I$, and the steps involving these are inserted in
order to match the middle composite $M\xto{\ep} I \xto{e} M$
in~\eqref{eq:intro-antipode}.

With this as preview and motivation, we now move on to a somewhat more
detailed description of our type theory.

\subsection{Generalized Sweedler notation}
\label{sec:gen-sweed}

The idea behind our type theory is to formalize and generalize the
informal \emph{Sweedler notation} that is common in coalgebra theory.
A coalgebra is a comonoid in the category of vector spaces, i.e.\ a
vector space $C$ with a comultiplication $\comult : C\to C\otimes C$
that is coassociative and counital.  Since the elements of the tensor
product $C\otimes C$ are finite sums of generating tensors
$c\s1 \otimes c\s2$, we can write
\[ \comult(c) = \sum_i c^i\s1 \otimes c^i\s2.
\]
Sweedler's notation is to omit the index $i$, and sometimes also the
symbol $\sum$, obtaining
\[ \comult(c) = c\s1 \otimes c\s2.
\]
(This is therefore a sort of ``dual'' to the ``Einstein summation
convention'' for products of tensors, which also has a type-theoretic
formalization known to physicists as ``abstract index
notation''~\cite{pr:spinors}.)  For instance, coassociativity can then
be expressed as
\[ c\s1 \otimes c\ss21 \otimes c\ss22 = c\ss11 \otimes c\ss12 \otimes c\s2
\]
(which is then often written as $c\s1 \otimes c\s2 \otimes c\s3$),
and counitality, for a counit $\ep$, can be expressed as
\[ c\s1 \cdot \ep(c\s2) = c = \ep(c\s1) \cdot c\s2
\]
where $\cdot$ denotes scalar multiplication.

We generalize this notation to apply to any morphism
$f:A \to B\otimes C$ in any monoidal category whose codomain is a
tensor product.  Of course we now have to notate \emph{which} morphism
we are talking about, so instead of $c\s1\otimes c\s2$ we write
$f\s1(x) \otimes f\s2(x) : B\otimes C$, where $x:A$ is a formal
variable.
Note that in Sweedler's original notation for coalgebras \emph{in
  vector spaces}, the element $\comult(c)$ really is a sum of
generating tensors $c\s1 \otimes c\s2$, whereas now we are in an
arbitrary monoidal category so that this doesn't even make sense.
Nevertheless, we can still use a similar \emph{formal} notation
governed by appropriate rules, just as in type theory we use variables
and terms assigned to objects of the category as ``types'' even though
the objects of an arbitrary category may not actually have
``elements''.

However, as a nod towards this extra formality (and also an extra step
in the ``types are like sets'' direction), we stop using the symbol
$\otimes$ and write instead a pair $(f\s1(x), f\s2(x)):(B,C)$
(although see \cref{sec:coalg}).  We also have to incorporate an
analogue of ``scalar multiplication'', which in a general monoidal
category means the unit isomorphism $C\otimes I \cong C$; for
consistency we collect all the ``scalars'' in a separate tuple
separated by a bar.  Thus, for instance, the counitality of a
comultiplication $\comult : C\to C\otimes C$ with counit
$\ep : C\to I$ in an arbitrary monoidal category will be expressed as
\[ (\comult\s1(c) \mid \ep(\comult\s2(c))) = c =
  (\comult\s2(c) \mid \ep(\comult\s1(c))).
\]
And the composite of $f:A \otimes B \to C\otimes D$ with
$g:E\otimes D\to F\otimes G$ along the object $D$ is
\[
  \inferrule{x:A, y:B \types (f\s1(x,y),f\s2(x,y)):(C,D) \\
    z:E,w:D \types (g\s1(z,w),g\s2(z,w):(F,G)) }
  {x:A, y:B, z:E \types (f\s1(x,y),g\s1(z,f\s2(x,y)),g\s2(z,f\s2(x,y))):(C,F,G)}
\]
Note that in contrast to intuitionistic linear logic, the variables in
a context are not literally treated ``linearly'' in our terms, since
they can occur multiple times in the multiple ``components'' of a map
$f$.  Instead, the ``usages'' of a variable are controlled by the
codomain arity of the morphisms applied to them.

One further technical device is required to deal with morphisms having
nullary domain.  Suppose we have $f:I\to B\otimes C$, written in our
type theory as $\types (f\s1,f\s2):(B,C)$, and we compose/tensor it
with itself (and apply a symmetry) to get a morphism
$I \to B\otimes B\otimes C\otimes C$.  We would na\"ively write this
as $ \types (f\s1,f\s1,f\s2,f\s2) : (B,B,C,C)$, but this is ambiguous
since we can't tell which $f\s1$ matches which $f\s2$.  Thus, we
disambiguate the possibilities by annotating the terms in pairs, such
as $() \types (f\s1,f'\s1,f\s2,f'\s2)$ or
$() \types (f\s1,f'\s1,f'\s2,f\s2)$.

The ``variable-binding notation'' $(x,\abs{A}{x})$ for duality that we
mentioned earlier is just syntactic sugar for a special case of this:
we use variables like $x$ as the ``labels'' for unit morphisms
$\eta_A : () \to (A,A^*)$ in place of primes, writing $(x,\abs{A}{x})$
instead of $((\eta_A)^x\s1,(\eta_A)^x\s2)$.  Of course, this requires
that the ``variables'' used as labels of this sort are disjoint from
those occurring in the actual context.  The operator $\eval{}{}$ is
just an infix notation for the counit $\ep:(A^*,A)\to ()$.

Another bit of syntactic sugar is that if each type is equipped with
at most one specified comonoid structure, then there is no ambiguity
in reverting back to Sweedler's original notation $(x\s1,x\s2)$
instead of $(\comult\s1(x),\comult\s2(x))$.  For instance, the axiom
of a Frobenius algebra can be written as
\begin{equation}
  (x\s1, x\s2 y) = ((xy)\s1,(xy)\s2) = (x y\s1, y\s2)\label{eq:frobenius}
\end{equation}
while the principal axiom of a bialgebra is
\begin{equation}
  ((xy)\s1,(xy)\s2) = (x\s1 y\s1, x\s2 y\s2).\label{eq:bialg}
\end{equation}

Note that if all types have such a comonoid structure \emph{and} all
morphisms are comonoid morphisms, which in type-theoretic notation
means $f(x)\s{i} = f(x\s{i})$ and $\ep(f(x)) = \ep(x)$, then the
monoidal structure automatically becomes cartesian.  In ordinary type
theory for cartesian monoidal categories, the diagonal (i.e.\ the
comultiplication) is represented by literally duplicating a variable
$(x,x)$; thus our subscripting of variables $(x\s1,x\s2)$ can be
viewed as a minimal modification of this to deal with situations where
the comultiplication is not literally a cartesian diagonal.  (Indeed,
the comultiplication of a coalgebra is often viewed as a non-cartesian
substitute for the diagonal, e.g.\ in the theory of quantum groups.)
We can similarly regard the counit $\ep$ as a ``discarding''
operation, in which case I like to notate $\ep(x)$ by $\cancel{x}$ (as
was done in~\eqref{eq:intro-antipode-uniq}).

\subsection{An alternative notation for coalgebraists}
\label{sec:coalg}

The particular symbols we choose in the syntax of our type theory are,
of course, immaterial.  I personally find the pair notation
$(f\s1(x),f\s2(x))$ evocative and helpful, since I am more used to
thinking of sets than coalgebras.  However, a reader who is more used
to ordinary coalgebras in vector spaces may prefer to retain $\otimes$
as a formal symbol in place of a comma, and use a symbol like $\cdot$
for ``scalar multiplication'' in place of the separating bar $\mid$.
For instance, in this notation the axioms~\eqref{eq:frobenius}
and~\eqref{eq:bialg} become
\begin{mathpar}
  x\s1 \otimes x\s2 y = (xy)\s1 \otimes (xy)\s2 = x y\s1 \otimes y\s2\and
  (xy)\s1 \otimes (xy)\s2 = x\s1 y\s1 \otimes x\s2 y\s2.
\end{mathpar}
while the calculation~\eqref{eq:intro-antipode-uniq} becomes
\[
  \inv{x} = \inv{x}\,e = \inv{x\s1}\, e \cdot \ep(x\s2) =
  \inv{x\s1}\, x\s2 \,\invt{x\s3} = e\, \invt{x\s2} \cdot \ep(x\s1) =
  e\,\invt{x} = \invt{x}.
\]
A reader who prefers this notation is free to use it instead.
(However, one should probably then choose a different notation for the
tensor product \emph{types} discussed in \cref{rmk:tensors}.)

\subsection{Some pluralistic remarks}
\label{sec:vistas}

Since many readers will be more familiar with other syntaxes for
monoidal categories such as intuitionistic linear logic or string
diagrams, it is worth saying a few words about their complementary
relationship to our type theory, and in particular how and why they
can all coexist.

Intuitionistic linear logic was discussed briefly above; from our
present point of view, it has the disadvantages of a heavy and
unintuitive ``match'' syntax for decomposing elements of tensor
products.  However, it also has definite advantages,\footnote{I am
  indebted to Matt Oliveri for these points.} such as an
asymptotically more concise syntax: the present type theory involves
much duplication of terms, since the arguments of a function are
repeated in every component of its output.  This matters little for
the examples we consider here, but might become more important if it
were ever implemented at scale in a computer proof assistant.
Intuitionistic linear logic also cleanly extends to a syntax for
\emph{closed} monoidal categories, whereas with our current type
theory this seems difficult.

String diagrams (a.k.a.\ graphical calculus) are another well-known
and very powerful tool for working with structures in many kinds of
monoidal categories, which also have both strengths and weaknesses.
String diagrams are at their best when dealing with structures whose
axioms ``don't change the topology'', such as dual pairs and Frobenius
algebras, since in this case many proofs are simply topological
deformations.  For structures whose axioms do change the topology,
such as bialgebras and Hopf algebras, string diagrams are still
useful, but in such cases only some of the proof steps can be reduced
to simple topological deformations: generally those steps involving
pure \emph{naturality} conditions.

By contrast, our type theory does not ``see'' the topological nature
of dual pairs and Frobenius algebras.  However, it represents most
naturality conditions as \emph{syntactic identities}, making them
completely invisible; for instance, the two middle steps in
\cref{fig:cyclicity} disappear entirely
in~\eqref{eq:intro-antipode-uniq}.  Moreover, it leverages a different
intuition, allowing us to think of objects of an arbitrary monoidal
category as ``sets'' with ``elements'' (modulo certain ``linearity''
restrictions).  This seems to be particularly useful for coalgebraic
and Hopf-type structures --- perhaps unsurprisingly, given the origin
of our syntax in Sweedler notation.

Another difference is that string diagrams incarnate categorical
duality in an obvious way --- by simply rotating or reflecting a
string diagram --- while our type theory breaks this duality, treating
inputs and outputs differently.  One might consider this an advantage
of string diagrams.  But breaking duality in the syntax is part of
what gives us the above advantages (notably the view of types as
sets), and it also means that duality becomes a nontrivial and useful
technique: in a situation that is not self-dual, we can choose which
orientation of the type theory is most convenient.  For instance, in
\cref{sec:weak-bimonoids} we will use our type theory to show that the
antipode of a (weak) Hopf monoid is a monoid anti-homomorphism; it
then follows by duality that it is also a comonoid anti-homomorphism,
although a direct proof of the latter in our syntax would be less
intuitive.

There is no reason to expect any one notation for symmetric monoidal
categories to be the best in all situations: each has applications to
which it is better-suited and others to which it is not as well
suited.  The new type theory presented here makes no claim to supplant
string diagrams or intuitionistic linear logic; rather it is an
alternative that may be more convenient in \emph{some} applications.
Only time and experience can render a final verdict on the usefulness
of a notation, as well as its generalizations and limitations ---
e.g.\ to what extent can the type theory presented in this paper be
generalized to other kinds of monoidal categories, such as closed,
$\ast$-autonomous, braided, ribbon, and planar monoidal categories,
indexed monoidal
categories, 
(symmetric) monoidal
bicategories, 
and so on?  My primary hope for this paper is to begin a conversation
about type theories in this family and their potential uses.

\subsection{Acknowledgements}
\label{sec:acknowledgements}

I would like to thank Dan Licata, Robin Cockett, Peter LeFanu
Lumsdaine, and Matt Oliveri for useful conversations.

\section{Props}
\label{sec:props}

To simplify and clarify the semantics of our type theory,
we will interpret it in a categorical structure that reflects the
type-theoretic distinction between judgmental and type operations.  To
explain what this means, recall that ordinary cartesian type theory is
often interpreted in categories with finite products; but in this case
both the judgmental comma (in a context $(x:A, y:B, z:C)$) and the
product type operation (in a product type $(A\times B)\times C$) are
interpreted by the same categorical operation (cartesian product
types), which can cause confusion.  A more direct semantics of
cartesian type theory maintains this distinction by using a
``cartesian multicategory'', allowing the judgmental comma to be
interpreted by the ``categorical comma'', i.e.\ the concatenation that
forms a list of objects to be the domain of a morphism in a
multicategory.

For type theories like classical linear logic, which allow multiple
types in both domain and codomain but use commas on the left and right
of the turnstile to represent different tensor products, the
appropriate ``multicategorical'' structure is called a
``polycategory''~\cite{szabo:polycats} (the corresponding ``monoidal''
version being a linearly
distributive~\cite{cs:wkdistrib,bcst:natded-coh-wkdistrib} or
$\ast$-autonomous~\cite{barr:staraut,barr:staraut-ll} category).  In
our case, where the two tensor products are the same, the appropriate
structure is called a \emph{prop}.

\begin{defn}\label{defn:prop}
  A \textbf{prop} $\cP$ consists of
  \begin{enumerate}
  \item A set $\ob(\cP)$ of \textbf{objects}, and\label{item:prop1}
  \item A symmetric strict monoidal category (that is, a symmetric
    monoidal category whose associators and unitors are identities)
    whose underlying monoid of objects is freely generated by
    $\ob(\cP)$.\label{item:prop2}
  \end{enumerate}
\end{defn}

(The original Adams--MacLane~\cite{maclane:cat-alg} definition of prop
had only one object; thus our ``props'' are sometimes called ``colored
props''.)

We write the objects of the monoidal category in~\ref{item:prop2} as
finite lists $(A,B,\dots,Z)$ of objects of the prop (i.e.\
$A,B,\dots,Z\in\ob(\cP)$).  The monoidal structure is given by
concatenation of lists; the unit object is the empty list $()$.

We now summarize the relationship between props and symmetric monoidal
categories.

\begin{defn}\label{defn:prop-tensor}
  A \textbf{tensor product} $A\otimes B$ of objects in a prop is an
  object together with an isomorphism
  \[ (A,B) \toiso (A\otimes B) \]
  in the monoidal category of \cref{defn:prop}\ref{item:prop2}.
  Similarly, a \textbf{unit} is an object $I$ with an isomorphism
  \[ () \toiso (I) .
  \]
  A prop is called \textbf{representable} if it has a unit and any
  pair of objects has a tensor product.
\end{defn}

A \textbf{morphism of props} $\omega : \cP\to\cQ$ is a function
$\omega_0 : \ob(\cP) \to \ob(\cQ)$ together with a \emph{strict}
symmetric monoidal functor whose action on objects is obtained by
applying $\omega_0$ elementwise to lists.

\begin{thm}\label{thm:smc-prop}
  The category of symmetric monoidal categories and strong symmetric
  monoidal functors is equivalent to the subcategory of representable
  props (and all morphisms between them).
\end{thm}
\begin{proof}[Sketch of proof]
  Every symmetric monoidal category \cC has an underlying prop $U\cC$
  with the same objects, and in which a morphism
  $(A,\dots,B) \to (C,\dots, D)$ is a morphism
  $A\otimes \cdots \otimes B \to C\otimes\cdots \otimes D$ in \cC.
  This construction $U$ is functorial on strong symmetric monoidal
  functors, and the props in its image are representable.  (In fact,
  it is the functorial strictification.)  The functor $U$ is faithful
  since the action of $f:\cC\to\cD$ on objects and arrows is preserved
  in $U f$, while the coherence constraints of $f$ are recorded in the
  action of $U f$ on the isomorphisms from \cref{defn:prop-tensor}.
  Similarly, the functor $U$ is full, since the action of a morphism
  $U \cC\to U\cD$ on the isomorphisms from \cref{defn:prop-tensor}
  induces symmetric monoidal constraints on its underlying ordinary
  functor.  Finally, every representable prop induces (by choosing
  tensor products and a unit) a symmetric monoidal structure on its
  category of unary and co-unary morphisms, and it is isomorphic to
  the underlying prop thereof.
\end{proof}

We can say more: this subcategory is an injectivity-class, and the
corresponding category of \emph{algebraic} injectives (objects
equipped with chosen lifts against the generating morphisms, and maps
that preserve the chosen lifts) is equivalent to the category of
symmetric monoidal categories and \emph{strict} symmetric monoidal
functors.  In particular, the latter is monadic over the category of
props.

The fundamental ``initiality theorem'' for semantics of our type
theory, which we prove in \cref{sec:semantics}, is that the ``term
model'' is the prop freely generated by some input data.
Following~\cite{bcr:props}, we will call this input data a
\emph{signature}.

\begin{defn}
  A \textbf{signature} \cG is a set of objects together with a set of
  arrows, each assigned a domain and codomain that are both finite
  lists of objects.
\end{defn}

\begin{thm}\label{thm:props-monadic}
  The category of props is monadic over the category of signatures.
\end{thm}
\begin{proof}
  Just like the proof in~\cite[Appendix A.2]{bcr:props} for the one-object case.
\end{proof}

Thus, every prop has a \emph{presentation} as a coequalizer of a pair
of maps between free props.  In \cref{sec:presentations-props} we will
extend our type theory to construct ``presented props'' as well,
allowing equational reasoning as in the examples from
\cref{sec:ttmon}.

\begin{rmk}\label{rmk:tensors}
  The fact that props allow multiple objects in both domains and
  codomains means that we rarely need to talk about actual tensor
  product types $A\otimes B$ with semantics in tensor product objects
  (\cref{defn:prop-tensor}).  For this reason, we will not include
  such types in our formal system.  However, we note that if
  necessary, they can be added quite easily: because
  \cref{defn:prop-tensor} simply asserts objects, morphisms, and
  equations (rather than a unique factorization property), it can
  essentially be ensured as a special case of a prop presentation.
  Thus we need only add some generating terms and axioms to our type
  theory, for any pair of objects $A,B$ whose tensor product we need
  to talk about:
\begin{mathpar}
  x:A, y:B \types \pair{x}{y}:A\otimes B\and
  p : A\otimes B \types (\pi\s1(p), \pi\s2(p)) : (A,B) \and
  (x,y) = (\pi\s1\pair x y, \pi\s2\pair x y)\and
  p = \pair{\pi\s1(p)}{\pi\s2(p)}.
\end{mathpar}
Note how similar this is to the treatment of cartesian products in
cartesian type theory with pairing and projection operations.
\end{rmk}

\section{On the admissibility of structural rules}
\label{sec:uniqder}

In general, type theories consist of \emph{rules} for deriving
\emph{judgments} about \emph{terms}.  The most common judgments are
\emph{typing judgments} (that a term belongs to a type, or in our case
a tuple of terms belong to a tuple of types) and \emph{equality
  judgments} (that two terms --- or tuples of terms --- are equal).  A
tree of rules ending with a judgment is called a \emph{derivation} of
that judgment, and the categorical structure presented by a type
theory is built out of the \emph{derivable} (or \emph{valid})
judgments.

Now in proving that this categorical structure does in fact have the
desired universal property, it is very useful if we arrange the type
theory so that every derivable judgment has a \emph{unique}
derivation.  The reason for this is that we want a morphism in our
categorical semantics to be determined by a \emph{term itself}, not by
a choice of derivation of that term; but the natural way to prove the
desired universal property (a.k.a.\ the ``initiality theorem'' for
that type theory) is by induction \emph{over derivations}.  Thus, if
the same term can arise from multiple derivations, proving this
universal property requires an extra step of proving that this choice
is immaterial (i.e.\ that any two derivations of the same term
determine the same morphism in the semantics).  This step is tricky
and often omitted in the literature, leading to incomplete proofs.  It
becomes even trickier when considering higher-categorical semantics,
in which the morphisms determined by two derivations of the same term
may be only \emph{isomorphic} rather than equal.

The choice that terms should have unique derivations essentially
requires that nearly all structural rules should be admissible rather
than primitive.  (Recall that an \emph{admissible rule} is one that is
\emph{not} asserted as part of the specification of the type theory,
but for which we can prove after the fact that whenever we have
derivations of its premises we can construct a derivation of its
conclusion --- usually by inductively traversing and modifying the
given derivations of its premises.  The \emph{structural rules} are,
roughly speaking, those that correspond to the operations of the
categorical structure used as the semantics: composition in a
category, permutation of domain lists in a symmetric multicategory,
and permutation of both domains and codomains in a polycategory or
prop.)

The reason this requirement arises is that structural rules generally
have to satisfy equations that are ``tautological'' in their action on
terms.  For instance, composing $f:A\to B$ with $g:B\to C$ and
$h:C\to D$ in the two possible ways (semantically, $h\circ (g\circ f)$
and $(h\circ g)\circ f$) produces the same term:
\begin{mathpar}
\inferrule*{
  \inferrule*{x:A \types f(x): B \\
    y:B \types g(y):C}{x:A \types g(f(x)):C}\\ z:C \types h(z):D}
  {x:A \types h(g(f(x))) :D}\and
\inferrule*{x:A \types f(x): B \\
  \inferrule*{y:B \types g(y):C \\
    z:C \types h(z):D}{y:B \types h(g(y)):D}}
  {x:A \types h(g(f(x))) :D}.\and  
\end{mathpar}
Thus, if composition were a primitive rule, this term would have two
distinct derivations.  But if composition (i.e.\ substitution) is an
admissible rule, then we can \emph{prove} that the derivations
\emph{constructed} by applying it in these two different ways
\emph{turn out to be} the same.  Similarly, if permutation were
primitive, then for any $f:(A,B) \to C$ we would have two (in fact,
infinitely many) derivations of the same term:
\begin{mathpar}
  \inferrule{ }{x:A ,y:B \types f(x,y):C} \and
  \inferrule{\inferrule{\inferrule{ }{x:A ,y:B
        \types f(x,y):C}}{y:B,x:A \types f(x,y) :C }}
  {x:A ,y:B \types f(x,y):C}
\end{mathpar}
whereas if permutation (a.k.a.\ ``exchange'') is admissible, then its
functoriality as an operation on derivations can be proven.

Type theorists know how to make a rule admissible: we build ``just
enough'' of it into the primitive rules.  For instance, if we
introduce generating morphisms such as $f:(A,B) \to C$ and $g:C\to D$
as simple axioms (i.e.\ rules with no premises):
\begin{mathpar}
  \inferrule{ }{x:A, y:B \types f(x,y):C}\and
  \inferrule{ }{z:C \types g(z):D}
\end{mathpar}
then there would be no way to construct derivations of composites such as
\begin{equation}
  \inferrule{x:A, y:B \types f(x,y):C \\
    z:C \types g(z):D}{x:A, y:B \types g(f(x,y)) : D}\label{eq:primcomp}
\end{equation}
except by using a \emph{primitive} composition/substitution rule.
Therefore, we instead introduce each generating morphism in ``Yoneda
style'' by allowing ourselves to \emph{postcompose} any given term(s)
with it.  For instance, in cartesian type theory we introduce
generators with rules such as
\begin{mathpar}
  \inferrule{\Gamma \types a:A \\ \Gamma\types b:B}{\Gamma\types f(a,b):C}\and
  \inferrule{\Gamma\types c:C}{\Gamma\types g(c):D}
\end{mathpar}
for arbitrary contexts $\Gamma$ and terms $a,b,c$.
Now~\eqref{eq:primcomp} can be obtained without a primitive substitution rule:
\begin{equation}
  \inferrule*{\inferrule*{\inferrule*{ }{x:A,y:B \types x:A}\\
      \inferrule*{ }{x:A,y:B \types y:B}}
    {x:A,y:B \types f(x,y):C}}
  {x:A,y:B \types g(f(x,y)):D}\label{eq:admcomp}
\end{equation}
In categorical terms, the point is that we can build the free category
on a graph either by freely adding all binary (and nullary) composites
and then quotienting by the relation of associativity, or we can avoid
the need to quotient at all by defining the morphisms as composable
lists of generating arrows, where a ``list'' is defined inductively as
either an empty list or a list postcomposed by a generating arrow ---
that is, we enforce right-associated composites
$k\circ (h\circ (g\circ (f\circ \id)))$.

This technique is trickier in non-cartesian type theories, since we
cannot keep the same context all the way through the derivation.  That
is, in a cartesian monoidal category we can start~\eqref{eq:admcomp}
with the ``identity'' or ``axiom'' rules $x:A,y:B \types x:A$ and
$x:A,y:B \types y:B$, corresponding categorically to the projections
$A\times B\to A$ and $A\times B\to B$; but in a non-cartesian monoidal
category such projections do not exist.  Thus, we need to concatenate
contexts as we go down the derivation tree.  For instance, the
generator rule for $f:(A,B)\to C$ must be something like
\begin{equation}
  \inferrule{\Gamma\types a:A \\ \Delta\types b:B}
  {\Gamma,\Delta \types f(a,b):C}.\label{eq:gen1}
\end{equation}

However, now we have a new problem: the exchange rule (permutations of
the domain).  In the cartesian case, we can make this admissible by
``propagating it up'' the entire derivation since the context remains
the same.  For instance, if we permute the inputs
of~\eqref{eq:admcomp} we would simply get
\begin{equation}
  \inferrule*{\inferrule*{\inferrule*{ }{y:B,x:A \types x:A}\\
      \inferrule*{ }{y:B,x:A \types y:B}}
    {y:B,x:A \types f(x,y):C}}
  {y:B,x:A \types g(f(x,y)):D}\label{eq:permcomp}
\end{equation}
But in the non-cartesian case this doesn't work.  We don't want to
assert a primitive exchange rule as this would break the ``terms have
unique derivations'' principle, so instead we build exchange into the
generator rule~\eqref{eq:gen1}.  Our first try might be something like
\[
  \inferrule{\Gamma\types a:A \\ \Delta\types b:B \\
    \si:\Gamma,\Delta \cong \Phi}{\Phi \types f(a,b):C}
\]
where $\si$ is an arbitrary permutation.  But this too breaks the
``terms have unique derivations'' principle, since a permutation of
types \emph{within} $\Gamma$ or $\Delta$ could be obtained either as
part of $\si$ or by operating on the input derivations of
$\Gamma\types a:A$ and $\Delta\types b:B$.  Instead we have to build
in ``just enough'' exchange but not too much, by requiring $\si$ to be
not an arbitrary permutation but a \emph{shuffle}: a permutation of
$(\Gamma,\Delta)$ that preserves the \emph{relative} order of the
types in $\Gamma$ and in $\Delta$.  We will write
$\nShuf(\Gamma;\Delta)$ for the set of such shuffles.

This may seem overly technical, but the presence of such things as
shuffles need never be seen by the \emph{user} of the type theory.
Indeed, maintaining the ``terms have unique derivations'' principle is
precisely what \emph{allows} the user to work only with terms,
ignoring the details of the derivations.

\section{The type theory for free props}
\label{sec:technical-details}

Let \cG be a signature; we will define a type theory that presents the
free prop on \cG.  Our general typing judgment will be of the form
shown in \cref{fig:judgments}.  For reasons to be explained later, we
annotate some of the types in the consequent of each judgment with a
superscript star, $\actv{B}$, and call them \emph{active}; we write
$\maybeactv{B}$ to mean that $B$ might be active.
\begin{figure}
  \centering
  \small
  \begin{align*}
    x_1:A_1, \dots ,x_m:A_m &\types (M_1,\dots,M_n \mid Y_1,\dots, Y_p)
    :(\maybeactv{B_1},\dots,\maybeactv{B_n}).\\
    x_1:A_1, \dots ,x_m:A_m &\types
   (M_1,\dots,M_n \mid Y_1,\dots, Y_p) = (N_1,\dots,N_n \mid Z_1,\dots, Z_q)
  :(B_1,\dots,B_n).
  \end{align*}
\caption{Judgments}
\label{fig:judgments}
\end{figure}

We use vector notation $\vec{M},\vec{A},$ etc.\ to indicate a list of
terms or types, so the judgments can be abbreviated as in
\cref{fig:judgments-abbr}, although this omits the information that
$\vec{M}$ and $\vec{N}$ must have the same length as $\vec{B}$ (while
the length of $\vec{Y}$ and $\vec{Z}$ is unrestricted).  If $\vec{Y}$
is empty, we write $(\vec{M}\mid )$ as simply $(\vec{M})$.  If
furthermore $\vec{M}$ and $\vec{B}$ have length 1, we omit the
parentheses, writing simply $M : B$.  When all the lists are empty, we
have $\types () : ()$, which will be valid (it represents the identity
morphism of the unit object).

\begin{figure}
  \centering
  \begin{mathpar}
    \Gamma \types (\vec{M} \mid \vec{Y}) :\maybeactv{\vec{B}}\and
    \Gamma \types (\vec{M} \mid \vec{Y}) = (\vec{N} \mid \vec{Z}) :\vec{B}.
  \end{mathpar}
\caption{Judgments, abbreviated}
\label{fig:judgments-abbr}
\end{figure}

In these judgments $A_i,B_j$ are types, $x_i$ are variables, and
$M_j,N_j,Y_k,Z_\ell$ are terms.  Here by a \emph{type} we simply mean
an object of our generating signature $\cG$, since there are no
type-forming operations in the theory.  There is nothing new in our
\emph{variables}; the reader who prefers de Bruijn indices is free to
think of them in that way, although since our syntax has no variable
binding\footnote{Recall that in the notation $(x,\abs{A} x)$ for the
  unit of a duality, $x$ is not actually a variable in this sense but
  rather a ``label''; we will introduce labels formally in a moment.}
the usual subtleties of capture-avoidance are irrelevant.

The \emph{terms} are defined inductively by the rules shown in
\cref{fig:terms}.  Recall that in general our typing judgments involve
a \emph{list} of such terms, one for each type in the codomain,
together with a list of scalar terms.  The terms defined in
\cref{fig:terms} do not yet have any types, but they are
``well-scoped'' by definition: they come with a context and only use
variables occurring in that context.  We subscript only applications
of functions with greater than unary codomain, and as noted in
\cref{sec:introduction}, we annotate each occurrence of a morphism
with nullary domain and positive-ary codomain with some element of an
infinite alphabet of symbols \fA (such as $','',''',\dots$, or
$1,2,3,\dots$).

\begin{figure}
  \centering
  \begin{mathpar}
  \inferrule{(x:A) \in \Gamma}{\Gamma \types x \preterm}\and
  \inferrule{f\in \cG(;B_1,\dots,B_n)\\
    \fa\in\fA \\ n \ge 2 \\ 1\le k \le n
  }{\Gamma \types f^\fa\s k \preterm}\and
  \inferrule{f\in \cG(;B)\\
    \fa\in\fA
  }{\Gamma \types f^\fa \preterm}\and
  \inferrule{f\in \cG(;)
  }{\Gamma \types f \preterm}\and
  \inferrule{\Gamma \types M_1\preterm \\ \dots \\ \Gamma \types M_m\preterm\\\\
    f\in \cG(A_1,\dots,A_m; B_1,\dots, B_n)\\
    m\ge 1 \\ n \ge 2 \\ 1 \le k \le n
  }{\Gamma \types f\s k(M_1,\dots,M_m) \preterm}\and
  \inferrule{\Gamma \types M_1\preterm \\ \dots \\ \Gamma \types M_m\preterm\\\\
    f\in \cG(A_1,\dots,A_m; B_1,\dots, B_n)\\
    m\ge 1 \\ n \le 1
  }{\Gamma \types f(M_1,\dots,M_m) \preterm}
\end{mathpar}
\caption{Terms}
\label{fig:terms}
\end{figure}

We will write $\vec{f}(\vec{M})$ for the list of all subscriptings of
the application of $f$ to the arguments $\vec{M}$.  For instance, if
$f:(A,B,C)\to (D,E)$ then $\vec f(\vec M)$ would be
\[(f\s1(M_1,M_2,M_3),f\s2(M_1,M_2,M_3)).\]
More generally, for $f\in \cG(A_1,\dots,A_m;B_1,\dots,B_n)$ we have
\[
  \vec{f}(\vec{M}) =
  \begin{cases}
    (f^\fa\s1 ,\dots, f^\fa\s n) &\qquad m=0,\, n\ge 2\\
    f^\fa &\qquad m=0,\, n= 1\\
    f &\qquad m=0,\, n= 0\\
    (f\s1(\vec{M}),\dots,f\s n(\vec{M})) &\qquad m\ge 1,\, n\ge 2\\
    f(\vec{M}) &\qquad m\ge 1,\, n\le 1.
  \end{cases}
\]
In the first and second case, we also write $\vec{f}^\fa$ to notate
the label \fa.  In the third case, we allow ourselves to write
$\vec{f}^\fa$ to mean simply $f$ (discarding the label).  Finally,
when $n\le 1$ we allow ourselves to write $f\s1$ or $f\s1(\vec{M})$ to
mean $f$ or $f(\vec{M})$ respectively.

We define the ``simultaneous substitution'' of a list of terms
$\vec M$ for a list of variables $\vec x$ in the usual way, as shown
in \cref{fig:sub}.

\begin{figure}
  \centering
  \begin{alignat*}{2}
  x_k[M_1,\dots,M_n/x_1,\dots,x_n] &= M_k\\
  y[\vec M/\vec x] &= y &\quad (y\notin \vec x)\\
  f^\fa\s k[\vec M/\vec x] &= f^\fa\s k\\
  f^\fa[\vec M/\vec x] &= f^\fa\\
  f[\vec M/\vec x] &= f\\
  f\s k(N,\dots,P)[\vec M/\vec x] &=
  f\s k(N[\vec M/\vec x],\dots,P[\vec M/\vec x])\\
  f(N,\dots,P)[\vec M/\vec x] &=
  f(N[\vec M/\vec x],\dots,P[\vec M/\vec x])
\end{alignat*}
\caption{Simultaneous substitution into terms}
\label{fig:sub}
\end{figure}

We now move on to the rules governing our typing judgment.  In
\cref{sec:uniqder} we argued that by incorporating Yoneda-style
generator rules and shuffles, we can make composition and exchange
admissible and thereby ensure that any judgment has a unique
derivation.  In the case of props, we also want to make the monoidal
structure admissible (since it satisfies strict associativity and
interchange laws that we would otherwise have to assert as judgmental
equalities).  In particular, for morphisms $f:A\to B$ and $g:C\to D$
we would like the judgment
\[ x:A , y:C \types (f(x), g(y)) : (B,D)
\]
to have a unique derivation.  Symmetry suggests that this unique
derivation cannot apply $f$ first and then $g$ or vice versa.  Thus,
we replace the generator rule by a ``multi-generator'' rule allowing
only a one-step derivation
\[ \inferrule{x:A,y:C \types (x,y):(A,C)}{x:A,y:C \types (f(x),g(y)):(B,D)} \]
A first approximation to the general form of this rule is
\[ \inferrule*{
  \Gamma \types (\vec M,\dots,\vec N,\vec P \mid \vec Z)
  : (\vec A,\dots, \vec B,\vec C)\\
  f \in \cG(\vec A;\vec D)\\\cdots\\
  g \in \cG(\vec B;\vec E)
}{\Gamma \types (\vec f(\vec M),\dots,\vec g(\vec N),\vec P \mid \vec Z)
  : (\vec D,\dots, \vec E,\vec C).}
\]
However, if there are generators with nullary codomain, we need to
collect them into the scalar terms $\vec Z$.  Thus a second
approximation is
\[ \inferrule*{
  \Gamma \types (\vec M,\dots,\vec N,\vec P,\dots,\vec Q,\vec R \mid \vec Z)
  : (\vec A,\dots, \vec B,\vec C,\dots,\vec D, \vec E)\\
  f \in \cG(\vec A;\atleastone{\vec F})\\\cdots\\
  g \in \cG(\vec B;\atleastone{\vec G})\\\\
  h \in \cG(\vec C;)\\\cdots\\
  k \in \cG(\vec D;)
}{\Gamma \types
  \left(\vec f(\vec M),\dots,\vec g(\vec N),\vec R
    \;\middle|\;
    h(\vec P),\dots, k(\vec Q), \vec Z\right)
  : (\vec F,\dots, \vec G,\vec E).}
\]
(Here $\atleastone{\vec F}$ means that $\vec F$ contains at least one type.)
Eventually we will also incorporate shuffles, but we postpone that for now.
Let us consider instead how to prevent duplication of derivations.
In addition to our desired term
\begin{align}
  x:A, y:C &\types (f(x),g(y)):(B,D)\label{eq:prop-good-term}
  \intertext{we must also be able to write both of the following:}
  x:A, y:C &\types (f(x),y):(B,C)\label{eq:prop-goodish-term-1}\\
  x:A, y:C &\types (x,g(y)):(A,D).\label{eq:prop-goodish-term-2}
\end{align}
So how do we prevent ourselves from being able to apply the generator
rule again to the latter two, obtaining two more derivations of the
same morphism as~\eqref{eq:prop-good-term}?

There are different possible choices that one could make here,
potentially leading to different type theories and different ``normal
forms'' for the morphisms in a free prop.  The choice we will make is
to force ourselves to ``apply all functions as soon as possible''.
Thus, for instance, we forbid ourselves from applying $g$ to $y$
in~\eqref{eq:prop-goodish-term-1} because we \emph{could have} already
applied it to produce~\eqref{eq:prop-good-term}.  On the other hand,
we will still allow ourselves to apply $h:(B,C) \to E$
in~\eqref{eq:prop-goodish-term-1} to get
\[ x:A, y:C \types (h(f(x),y)):E
\]
because $h$ has $f$ as an input, hence could not have been applied at
the same time as $f$.

Making this precise is the purpose of designating some of the types in
the consequent as \textbf{active}; recall that we denote the active
types by $\actv{A}$.  If $\vec{A}$ is a list of types, we write
$\oneactv{\vec A}$ to mean that at least one of the types in $\vec A$
is active, $\allactv{\vec A}$ to mean that they are all active, and
$\zeroactv{\vec A}$ to mean that none of them are active.  As noted
previously, we write $\maybeactv{\vec A}$ to avoid specifying whether
or not any of the types are active.

The identity rule will make all types active, while the generator rule
makes only the outputs of the generators active.  We then restrict the
generator rule to require that at least one of the \emph{inputs} of
each generator being applied must be active in the premise; this means
that none of them could have been applied any sooner, since at least
one of their arguments was just introduced by the previous rule.
Thus, our desired derivation
\begin{mathpar}
  \inferrule*{\inferrule*{ }{x:A,y:C \types (x,y):(\actv{A},\actv{C})}}
  {x:A,y:C\types (f(x),g(x)):(\actv{B},\actv{D})}
\end{mathpar}
is allowed, while the undesired one
\begin{equation*}
  \text{\textquestiondown} \qquad
  \inferrule*{\inferrule*{\inferrule*{ }
      {x:A,y:C \types (x,y):(\actv{A},\actv{C})}}
    {x:A,y:C\types (f(x),y):(\actv{B},C)}}
  {x:A,y:C \types (f(x),g(y)):(B,D)} \qquad ?
\end{equation*}
is not allowed, since in the attempted application of $g$ the input
type $C$ is not active.  Thus our generator rule now becomes
\[ \inferrule*{
  \Gamma \types (\vec M,\dots,\vec N,\vec P,\dots,\vec Q,\vec R \mid \vec Z)
  : (\oneactv{\vec A},\dots, \oneactv{\vec B},
  \oneactv{\vec C},\dots,\oneactv{\vec D}, \maybeactv{\vec E})\\
  f \in \cG(\vec A;\atleastone{\vec F})\\\cdots\\
  g \in \cG(\vec B;\atleastone{\vec G})\\\\
  h \in \cG(\vec C;)\\\cdots\\
  k \in \cG(\vec D;)
}{\Gamma \types
  \left(\vec f(\vec M),\dots,\vec g(\vec N),\vec R
    \;\middle|\;
    \vec h(\vec P),\dots,\vec k(\vec Q), \vec Z\right)
  : (\allactv{\vec F},\dots, \allactv{\vec G},\zeroactv{\vec E}).}
\]
Of course, this rule can now never apply to generators with nullary
domain.  Since these can always be applied at the very beginning, we
incorporate them into the identity rule.  Thus the identity rule is
now
\[\inferrule{
      f \in \cG(;\atleastone{\vec{B}})\\\cdots\\
      g \in \cG(;\atleastone{\vec{C}})\\\\
      h \in \cG(;)\\\cdots\\
      k \in \cG(;)\\\\
      \fa,\dots,\fb\in \fA\text{ and pairwise distinct}
    }{\vec x:\vec A\types
      \left(\vec x,{\vec f}^{\fa},\dots,{\vec g}^{\fb} \,\middle|\, h,\dots,k\right)
      :(\allactv{\vec A}, \allactv{\vec B}, \dots,\allactv{\vec C}).}
\]
Note the labels on the terms with nullary domain and positive-ary
codomain, as promised.

Finally, if we want to make the exchange rule admissible, we have to
build permutations into the rules as well.  Each rule should add
exactly the part of a permutation that can't be ``pushed into the
premises''.  Because we've formulated the generator rule so that the
premise and conclusion have the same context, any desired permutation
in the domain can be pushed all the way up to the identity rule.
Thus, for the generator rule it remains to deal with permutation in
the codomain.

The freedom we have in the premises of the generator rule is to
(inductively) permute the types \emph{within} each list
$\vec A,\vec B,\vec C,\vec D,\vec E$, and also to block-permute the
lists $\vec A,\dots,\vec B$ and separately the lists
$\vec C,\dots,\vec D$ (with a corresponding permutation of the
generators $f,\dots,g$ and $h,\dots,k$).  (If we permuted the main
premise any more than this, it would no longer have the requisite
shape to apply the rule to.)  Permutations of $\vec C,\dots,\vec D$
don't do us any good in terms of permuting the codomain of the
conclusion, but we can push permutations of $\vec E$ directly into the
premise, and also a block-permutation of $\vec F,\dots,\vec G$ into a
block-permutation of $\vec A,\dots,\vec B$.

What remains that we have to build into the rule can be described
precisely by a permutation of $\vec F,\dots,\vec G,\vec E$ that (1)
preserves the relative order of the types in $\vec E$, and (2)
preserves the relative order of the \emph{first} types $F_1,\dots,G_1$
in the lists $\vec F,\dots,\vec G$.  That is, any permutation of
$\vec F,\dots,\vec G,\vec E$ can be factored uniquely as one with
these two properties followed by a block sum of a block-permutation of
$\vec F,\dots,\vec G$ with a permutation of $\vec E$.  (The choice of
the first types is arbitrary; we could just as well use the last
types, etc.)

There is no real need to allow ourselves to permute the scalar terms,
since semantically their order doesn't matter anyway.  But it is
convenient to allow ourselves to write the scalar terms in any order,
so we incorporate permutations there too.  The freedom in the premises
allows us to permute the term in $\vec Z$ arbitrarily, and also to
permute the terms $h,\dots,k$ among themselves; thus what remains is
precisely a shuffle.  The final generator rule is therefore the first
rule shown in \cref{fig:props}.

\begin{figure}
  \centering
  \begin{mathpar}
    \inferrule*{
      \Gamma \types (\vec M,\dots,\vec N,\vec P,\dots,\vec Q,\vec R \mid \vec Z)
      : (\oneactv{\vec A},\dots, \oneactv{\vec B},
      \oneactv{\vec C},\dots,\oneactv{\vec D}, \maybeactv{\vec E})\\
      f \in \cG(\vec A;\atleastone{\vec F})\\\cdots\\
      g \in \cG(\vec B;\atleastone{\vec G})\\\\
      h \in \cG(\vec C;)\\\cdots\\
      k \in \cG(\vec D;)\\\\
      \sigma : (\allactv{\vec F},\dots, \allactv{\vec G},\zeroactv{\vec E})
      \toiso \Delta \;\text{preserving activeness}\\\\
      \sigma \text{ preserves the relative order of types in } \vec E\\
      \sigma \text{ preserves the relative order of } F_1,\dots, G_1\\
      \tau \in \nShuf(h,\dots,k;\vec Z)
    }{\Gamma \types
      \left( \sigma\left(\vec f(\vec M),\dots,\vec g(\vec N),\vec R\right)
        \;\middle|\;
        \tau\left(h(\vec P),\dots,k(\vec Q), \vec Z\right)\right)
      : \Delta}
    \and
    \inferrule{
      f \in \cG(;\atleastone{\vec{B}})\\\cdots\\
      g \in \cG(;\atleastone{\vec{C}})\\\\
      h \in \cG(;)\\\cdots\\
      k \in \cG(;)\\\\
      \fa,\dots,\fb\in \fA\text{ and pairwise distinct}\\\\
      \sigma : (\allactv{\vec A}, \allactv{\vec B}, \dots,\allactv{\vec C})
      \toiso \Delta \;\text{preserving activeness}\\\\
      \sigma\text{ preserves the relative order of } B_1,\dots, C_1
    }{\vec x:\vec A\types
      \left(\sigma\left(\vec x,{\vec f}^{\fa},\dots,{\vec g}^{\fb} \right)
        \,\middle|\, h,\dots,k\right)
      :\Delta}
  \end{mathpar}
  \caption{Rules of the typing judgment}
  \label{fig:props}
\end{figure}

In the identity rule, the only useful freedom in the premises is to
block-permute the $\vec B,\dots,\vec C$.  Thus what remains is a
permutation that preserves the relative order of the first types
$B_1,\dots, C_1$.  Any permutation in the scalar terms can be pushed
into the premises, so we have the final rule shown second in
\cref{fig:props}.  Note that this also allows us to incorporate an
arbitrary permutation in the domain.

Having introduced the auxiliary notion of ``active types'' to ensure
that typing judgments have unique derivations, we now proceed to
eliminate it.  We start with the following observation.  The notion of
\emph{subterm} is defined as usual; we write $M\equiv N$ to mean that
$M$ and $N$ are syntactically the same term.

\begin{lem}\label{thm:distinct-terms}
  If $\Gamma \types (\vec M\mid\vec Z):\Delta$ is derivable, and some
  $M_i$ is a subterm of some $M_j$, then $i=j$ (hence
  $M_i\equiv M_j$).
\end{lem}
\begin{proof}
  By induction on the derivation.  In an application of the identity
  rule, the non-scalar terms have no proper subterms, so the
  assumption means that $M_i\equiv M_j$.  And since these terms are
  also uniquely identified by their label and subscript, we have
  $i=j$.

  For an application of the generator rule, all the non-scalar terms
  either occur verbatim in the main premise, or are of the form
  $f\s k(\vec M)$ (including $f(\vec{M})$ when $k=1$) where each $M_i$
  is a non-scalar term in the main premise and $|\vec{M}|\ge 1$.
  These two possibilities are mutually exclusive, and hence partition
  the terms into two classes that we call \emph{old} and \emph{new}
  respectively.

  The inductive hypothesis takes care of the case when both terms are
  old.  If a new term $f\s k(\vec M)$ is a subterm of an old term $N$,
  then each $M_i$ is a proper subterm of $N$, contradicting the
  inductive hypothesis.  This includes the case when an old term
  equals a new term; while if an old term $N$ is a proper subterm of a
  new term $f\s k(\vec M)$, then it must be a subterm of some $M_i$,
  also contradicting the inductive hypothesis.

  If one new term $f\s k(\vec M)$ is a proper subterm of another
  $g\s \ell(\vec N)$, then it must be a subterm of some $N_j$.  Hence
  each $M_i$ must be a proper subterm of $N_j$, contradicting the
  inductive hypothesis.

  Finally, suppose two new terms are syntactically equal,
  $f\s k(\vec M) \equiv g\s \ell(\vec N)$.  Then we must have $f=g$,
  $k=\ell$, and $\vec M \equiv \vec N$.  Note that $f=g$ means that
  $f$ and $g$ are the same function symbol (morphism in \cG), but not
  \emph{a priori} that they arise from the same generator
  \emph{application} in the rule (a given instance of the generator
  rule could apply the same generator more than once).  However,
  $\vec M \equiv \vec N$ and the inductive hypothesis ensure that they
  \emph{do} arise from the same generator application.  Together with
  $k=\ell$, this implies that they have the same place in the given
  judgment as well.
\end{proof}

We are primarily interested in applying \cref{thm:distinct-terms} in
the case of an ``improper subterm'' $M_i\equiv M_j$, but the stronger
hypothesis makes the induction go through more easily.

\begin{rmk}
  Semantically, it is not really necessary to label the morphisms in
  $\cG(;B)$ with a symbol $\fa\in\fA$, since tensor products of such
  morphisms are invariant under permutation (because the swap on the
  unit object is the identity morphism).  However, it would be
  significantly trickier to omit such labels syntactically.  In
  practice, we can leave them off informally and trust the reader to
  put them back if needed.
\end{rmk}

\begin{thm}\label{thm:prop-tad}
  If there is some assignment of activeness to the types in $\Delta$
  such that $\Gamma \types (\vec M\mid\vec Z):\Delta$ is derivable,
  then that assignment is unique, as is the derivation.
\end{thm}
\begin{proof}
  By the \emph{height} of an occurrence of a variable or function
  symbol in a term, we mean its maximum distance to a leaf node in the
  abstract syntax tree representing the term.  Thus the height of a
  variable is $0$, and the height of an occurrence of a function
  symbol is the least natural number strictly greater than the heights
  of the head symbols of all its arguments.
  Note that a nullary function symbol always has height $0$, while a
  function symbol applied to a positive number of variables alone has
  height $1$.

  We claim that in any derivable typing judgment, the terms associated
  to active types are precisely those non-scalar ones whose head
  symbol has maximum height.  The proof is by induction on
  derivations.  In the identity rule, all terms have height $0$ and
  all types are active.  Now consider the generator rule, and suppose
  inductively that the claim is true for the main premise, with
  maximal height $n$, say.  Then since each of the new function
  symbols introduced by the rule is applied to at least one term from
  an active type, which therefore has the maximal height $n$, it must
  have height $n+1$.  It follows that the new maximum height is $n+1$,
  and that these new symbols are precisely those of maximum height;
  but they are also precisely those associated to active types.  This
  proves the claim.

  It follows immediately that the terms uniquely determine the
  activeness of the types, since height is a syntactic invariant of
  the terms.  Moreover, we can tell from the terms which rule must
  have been applied last (if the maximum height is $0$, it must come
  from the identity rule; otherwise it must come from the generator
  rule) and which function symbols that rule must have introduced
  (those of maximum height).

  In the case of the generator rule, we can also conclude that the new
  terms must be precisely those whose head generator symbol has
  maximum height.  We divide these new terms into subsets that have
  the same generator symbol and arguments, so that each subset
  consists of terms $f\s i(\vec{M})$ for fixed $f$ and $\vec{M}$ but
  varying $i$.  By \cref{thm:distinct-terms}, there can be no more
  than one occurrence of a particular term $f\s i(\vec{M})$, so the
  subset of new terms determined by $f$ and $\vec{M}$ must consist of
  exactly the $m$ terms $f\s i(\vec{M})$ for $1\le i\le m$, where $f$
  has $m$ outputs.

  Now the ordering of these function symbol applications as
  $f,\dots,g$ and $h,\dots,k$ must be the order in which the
  corresponding $f\s1,\dots,g\s1$ and $h,\dots,k$ appear in the term
  list, since the permutations $\sigma$ and $\tau$ preserve those
  orders.  Then $\sigma^{-1}$ is uniquely determined by the fact that
  it must place all the outputs of $f$ first, and so on until all the
  outputs of $g$, then all the terms of non-maximum height in the same
  order that they were given in the conclusion.  Similarly,
  $\tau^{-1}$ is uniquely determined by the fact that it has to place
  $h,\dots,k$ first and the scalar terms of non-maximum height last,
  preserving internal order in each group.  Finally, this determines
  the main premise uniquely as well.

  The argument for the identity rule is similar, with no $\tau$ and
  with $\sigma^{-1}$ placing all the variables first in the order of
  the context.  Inductively, therefore, the entire derivation is
  uniquely determined.
\end{proof}

Note that we can regard this proof as an algorithm for
``type-checking'' a judgment without activeness annotations: first we
use the recursive height function on terms to calculate the
activeness, then we proceed as usual to recursively match against the
generator or identity rules.  Because of this theorem, in the future
we will omit the activeness labels.

\cref{thm:prop-tad} can be seen as giving a ``normal form'' for
morphisms in a free prop.  Intuitively, a composite of generating
morphisms and permutations is in normal form if each generator or
permutation is applied ``as soon as possible''; the inductive
definition of judgments with activeness makes precise exactly what
that means.

Another perspective on this involves string diagrams.  From the latter
perspective, a morphism in the free prop generated by $\cG$ can be
represented by an \emph{acyclic directed graph} whose vertices are
labeled by generators and whose edges are labeled by objects.  We
allow ``free edges'' one or both of whose ends do not connect to any
vertex, corresponding to the input and output objects of the morphism.
For instance, such a graph based on generators $f\in \cG(;G,C)$,
$g\in \cG(;D)$, $h\in \cG(C,A;F,E)$, $k\in \cG(B;)$,
$\ell\in \cG(E,D;H)$, and $m(G;)$ is shown in \cref{fig:string1}.
\begin{figure}
  \begin{subfigure}{0.5\textwidth}
  \centering
\begin{tikzpicture}[yscale=1.5]
  \node[vert](f) at (-.5,3) {$f$};
  \node[vert2](g) at (.9,1.8) {$g$};
  \node[vert2](h) at (0,2.1) {$h$};
  \node[vert2](k) at (1.4,1.1) {$k$};
  \node[vert2](m) at (-1,1.5) {$m$};
  \node[vert2](l) at (.5,1) {$\ell$};
  \draw[->] (f)-- node[ed] {$C$} (h);
  \draw[->] (h)-- node[ed] {$E$} (l);
  \draw[->] (g)-- node[ed] {$D$} (l);
  \draw[->] (1.5,3.3) -- node[ed] {$B$} (k);
  \draw[->] (.5,3.3) -- node[ed] {$A$} (h);
  \draw[->] (h) -- node[ed] {$F$} (-.5,0.3);
  \draw[->] (l) -- node[ed] {$H$} (.5,0.3);
  \draw[->] (f) -- node[ed,swap] {$G$} (m);
\end{tikzpicture}
\caption{Not longest-path layered}
\label{fig:string1}
\end{subfigure}
  \begin{subfigure}{0.5\textwidth}
  \centering
\begin{tikzpicture}[yscale=1.5]
  \node[vert](f) at (-.5,3) {$f$};
  \node[vert2](g) at (.9,3) {$g$};
  \node[vert2](h) at (0,2) {$h$};
  \node[vert2](k) at (1.5,2) {$k$};
  \node[vert2](m) at (-1,2) {$m$};
  \node[vert2](l) at (.5,1) {$\ell$};
  \draw[->] (f)-- node[ed,swap] {$C$} (h);
  \draw[->] (h)-- node[ed] {$E$} (l);
  \draw[->] (g)-- node[ed] {$D$} (l);
  \draw[->] (1.5,3.5) -- node[ed] {$B$} (k);
  \draw[->] (.5,3.5) -- node[ed] {$A$} (h);
  \draw[->] (h) -- node[ed] {$F$} (-.5,0.3);
  \draw[->] (l) -- node[ed] {$H$} (.5,0.3);
  \draw[->] (f) -- node[ed,swap] {$G$} (m);
\end{tikzpicture}
\caption{Longest-path layered}
\label{fig:string2}
\end{subfigure}
\caption{String diagrams for a morphism in a free prop}
\end{figure}
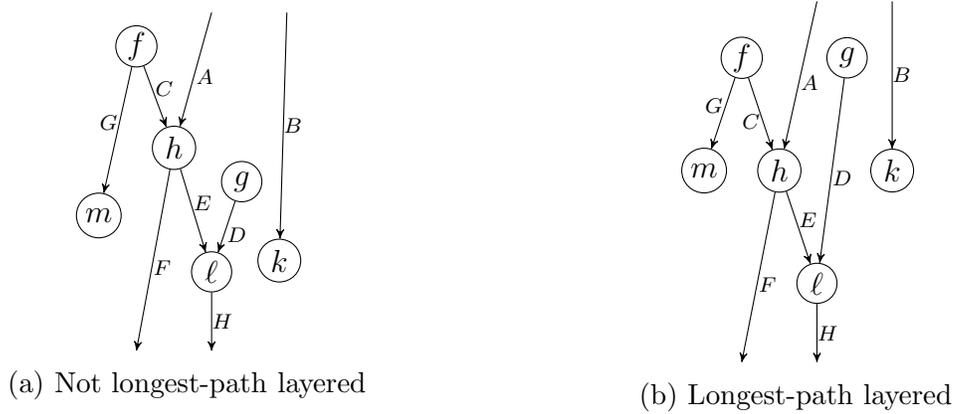
The corresponding term in our type theory is
\[ x:A, y:B \vdash (h\s1(f\s2,x), \ell(h\s2(f\s2,x),g) \mid m(f\s1),k(y)) : (F,H)
\]
Applying the algorithm of \cref{thm:prop-tad}, we first calculate the
heights of each generator:
\begin{mathpar}
  f:0\and
  g:0 \and
  h:1 \and
  k:1 \and
  \ell:2\and
  m:1
\end{mathpar}
Therefore, $H$ is the only active type, and the final rule application
must have been a generator rule applying only $\ell$:
\begin{equation}
  \infer{x:A, y:B \vdash (h\s2(f\s2,x),g, h\s1(f\s2,x) \mid m(f\s1),k(y))
    : (\actv{E},D,\actv{F})\\\\
  \ell \in \cG(E,D;H)}
{x:A, y:B \vdash (h\s1(f\s2,x), \ell(h\s2(f\s2,x),g) \mid m(f\s1),k(y))
  : (F,\actv{H})}.\label{eq:deriv3}
\end{equation}
In the premise of this rule, $E$ and $F$ are the active types since
$h$ has height 1.  But $k$ and $m$ also have height 1, so they must
also be applied by the generator rule leading to this judgment:
\begin{equation}
  \infer{x:A, y:B \vdash (f\s2,x,f\s1,y,g)
    : (\actv{C},\actv{A},\actv{G},\actv{B},\actv{D})\\\\
  h\in \cG(C,A;F,E)\\\\ m\in \cG(G;) \\ k\in \cG(B;)}
{x:A, y:B \vdash (h\s2(f\s2,x),g, h\s1(f\s2,x) \mid m(f\s1),k(y))
  : (\actv{E},D,\actv{F})}.\label{eq:deriv2}
\end{equation}
Now in the premise all the types are active and all terms have height
0, so it must have arisen from the identity rule:
\begin{equation}
  \infer{f\in \cG(;G,C) \\ g\in \cG(;D)}{x:A, y:B \vdash (f\s2,x,f\s1,y,g)
    : (\actv{C},\actv{A},\actv{G},\actv{B},\actv{D})}.\label{eq:deriv1}
\end{equation}
Thus we have a complete typing derivation:
\[\tiny
  \infer*{
  \infer*{
    \infer*{f\in \cG(;G,C) \\ g\in \cG(;D)}{x:A, y:B \vdash (f\s2,x,f\s1,y,g)
      : (\actv{C},\actv{A},\actv{G},\actv{B},\actv{D})}
    \\ h\in \cG(C,A;F,E)\\ m\in \cG(G;) \\ k\in \cG(B;)
  }
  {x:A, y:B \vdash (h\s2(f\s2,x),g, h\s1(f\s2,x) \mid m(f\s1),k(y))
    : (\actv{E},D,\actv{F})}
  \\ \ell \in \cG(E,D;H)
  }
  {x:A, y:B \vdash (h\s1(f\s2,x), \ell(h\s2(f\s2,x),g) \mid m(f\s1),k(y))
    : (F,\actv{H})}.
\]
This derivation corresponds to the redrawing of the string diagram
shown in \cref{fig:string2}, in which the vertices appear in
horizontal layers, with the edges always pointing down the page, and
each vertex placed in the highest layer possible.  The topmost layer
with $f$ and $g$ corresponds to the identity rule
application~\eqref{eq:deriv1}, the next layer with $h,m,k$ corresponds
to the generator rule application~\eqref{eq:deriv2}, and the final
layer with $\ell$ corresponds to~\eqref{eq:deriv3}.

The referee has pointed out that this is essentially the ``longest
path layering''~\cite{el:draw-digraph} of an acyclic directed graph
(modulo a suitable correction to deal with vertexless half-edges).
The proof of \cref{thm:prop-tad} can thus equivalently be interpreted
as computing a normal form for a morphism in a free prop by expressing
it as an acyclic directed graph and rearranging it according to its
longest path layering, and also making certain canonical choices
regarding the ordering of vertices along each layer.  It is natural to
expect that other canonical ways of drawing a directed graph might
correspond to other normal forms and hence other type theories for
free props, but we will not investigate this here.

It remains to consider the equality judgment.  We don't have a
traditional form of $\alpha$-conversion since we have no bound
variables as such, but as remarked in \cref{sec:introduction} the
labels $\fa\in\fA$ can sometimes be regarded as playing a similar
role, and in particular the choice of concrete labels must not matter.
We also need to impose invariance under permutation of the scalar
terms.  It may seem silly to have incorporated permutations in the
scalar terms earlier and yet quotient out by that freedom now, but
such an equality rule would be necessary even if we hadn't
incorporated any permutations to start with.  The paradigmatic case is
when we have two nullary scalar generators $f:() \to ()$ and
$g:()\to ()$, leading unavoidably to two distinct valid terms
\begin{mathpar}
   \types (\,\mid f,g) :\ec\and
   \types (\,\mid g,f) :\ec
\end{mathpar}
that must be equal in a monoidal category (since the monoid of
endomorphisms of the unit object is commutative).  It is probably no
coincidence that this is also the case where interesting things happen
upon categorification.

\begin{figure}
  \centering
  \begin{mathpar}
    \inferrule{\Gamma\types (\vec M \mid Z_1,\dots,Z_n) : \Delta \\
      \rho \in S_n \\
      \si : \fA \cong \fA}{\Gamma\types (\vec M \mid Z_1,\dots,Z_n) =
      (\vec M^\si \mid Z^\si_{\rho 1},\dots,Z^\si_{\rho n}) : \Delta}
  \end{mathpar}
  \caption{Rule of the axiom-free equality judgment}
  \label{fig:equality}
\end{figure}

Combining these two permutation rules, we obtain the rule shown in
\cref{fig:equality}.  When we consider \emph{presented} props in
\cref{sec:presentations-props}, with equational theories, we will need
the usual reflexivity, symmetry, transitivity, and congruence rules
for equality, but with only this rule we can omit them since
permutations are already a group.  And since we have no type forming
operations, there are no $\beta$ or $\eta$ rules.

This completes our \textbf{type theory for the free prop generated by \cG}.

\section{Constructing free props from type theory}
\label{sec:semantics}

We now proceed to show that our type theory has the structure of a
prop, beginning with the admissibility of exchange on the right.

\begin{prop}\label{thm:prop-symadm}
  If $\Gamma\types (\vec{M}\mid\vec{Z}) : \Delta$ is derivable and
  $\rho$ is a permutation of $\Delta$, then
  $\Gamma\types (\rho\vec{M}\mid\vec{Z}) : \rho\Delta$ is also
  derivable.  Moreover, this action is functorial.
\end{prop}
\begin{proof}
  This essentially follows from how we built the rules.  If the
  derivation ends with the identity rule, then we can compose $\rho$
  with the specified permutation $\sigma$ from that rule, and reorder
  the generators $f,\dots,g$ in the rule according to the order that
  $\rho\sigma$ puts them in.  If the derivation ends with the
  generator rule, then we similarly compose $\rho$ with $\sigma$,
  reorder the generators $f,\dots,g$, and inductively push the
  remaining part of the permutation (that acting on the non-active
  terms) into the main premise.  Functoriality follows immediately.
\end{proof}

For admissibility of composition/substitution, it seems helpful to
first prove the admissibility of a single-generator rule.  Note that
we formulate it with the domain of the generator at the \emph{end} of
the given codomain context.

\begin{prop}\label{thm:prop-onecutadm}
  If $\Gamma\types (\vec{M},\vec{N}\mid \vec{Z}): \Delta,\vec A$ is
  derivable and $f\in \cG(\vec A;\vec B)$, then
  $\Gamma\types (\vec{M},\vec{f}(\vec{N})\mid \vec{Z}): \Delta,\vec B$
  is derivable.  Moreover, if none of the types in $\vec A$ are active
  in the given derivation, then all of the types in $\Delta$ that are
  active in the given derivation are still active in the result.
\end{prop}
\begin{proof}
  If any of the types in $\vec A$ are active, we can simply apply the
  generator rule with $f$ as the only generator.  Otherwise, none of
  them were introduced by the final rule in the given derivation.  If
  that final rule was the identity rule, then $\vec A$ must be empty
  (since all types in the conclusion of the identity rule are active),
  so $f$ has nullary domain and we can just add it to that application
  of the identity rule.

  If the final rule in the given derivation was the generator rule,
  then $\vec A$ must also appear at the end of its main premise.  If
  none of the types in $\vec A$ are active therein, then we can
  inductively apply $f$ to that premise; by the second clause of the
  inductive hypothesis, this does not alter the activeness of the
  other types in the premise, so we can re-apply the generator rule.
  Finally, if at least one of the types in $\vec A$ \emph{is} active
  in the main premise, then we can add $f$ to the generator rule,
  applying it alongside all the other generators, since it satisfies
  the condition that at least one of its arguments be active.
  (Technically, this may require us to first permute the consequent of
  the main premise so that $\vec A$ appears before all the other
  non-inputs to the generator rule.  This is not a problem for the
  induction since in this case we are not actually using the inductive
  hypothesis at all.)  In all cases, the second claim of the lemma is
  obvious.
\end{proof}

Now by combining \cref{thm:prop-symadm} and \cref{thm:prop-onecutadm},
we can postcompose with a generator $f\in \cG(\vec A;\vec B)$ whose
domain types $\vec A$ appear anywhere in the consequent of a judgment
$\Gamma\types (\vec{M}\mid\vec{Z}):\Delta$, in any order.

\begin{prop}\label{thm:prop-cutadm}
  Substitution is admissible: if
  $\Gamma\types (\vec{M}\mid\vec{Y}) : \Delta$ and
  $\Delta \types (\vec{N}\mid\vec{Z}) : \Phi$ are derivable, then so
  is
  $\Gamma\types (\vec{N}[\vec{M}/\Delta] \mid
  \vec{Z}[\vec{M}/\Delta],\vec{Y}) : \Phi$.
\end{prop}
\begin{proof}
  We induct on the derivation of
  $\Delta \types (\vec{N}\mid\vec{Z}) : \Phi$.  If it comes from the
  identity rule, then we just have to compose
  $\Gamma\types (\vec{M}\mid\vec{Y}) : \Delta$ with some number of
  nullary-domain generators and permute its codomain; we do this one
  by one using \cref{thm:prop-onecutadm} and then
  \cref{thm:prop-symadm}.  Similarly, if it comes from the generator
  rule, we inductively compose with its main premise, then apply all
  of the new generators one by one using \cref{thm:prop-onecutadm}.
\end{proof}

As an example, suppose we want to compose the following terms:
\begin{gather}
  x:A,y:B \types (f\s1(y),k(g,f\s3(y)),f\s2(y) \mid h(x))
  : (C,D,E)\label{eq:ceg5a}
  \\\notag\\
  u:C,v:D,w:E \types (m(u,\ell\s2(w)),s,\ell\s1(w) \mid n(v))
  : (F,G,H)\label{eq:ceg5b}
\end{gather}
Here the generators are
\begin{mathpar}
  f:B \to (C,E,P)\and
  g:() \to Q\and
  h:A \to ()\and
  k:(Q,P) \to D\and
  \ell:E \to (H,R)\and
  m:(C,R) \to F\and
  n:D\to ()\and
  s:() \to G
\end{mathpar}
The heights are
\begin{mathpar}
  f=1\and g=0\and h=1 \and k=2\and \ell=1 \and m=2 \and n=1 \and s=0
\end{mathpar}
Thus, the final rule of the second derivation must apply $m$ only, so
our inductive job is to compose~\eqref{eq:ceg5a} with
\begin{equation}
  u:C,v:D,w:E \types (u,\ell\s2(w),s,\ell\s1(w) \mid n(v))
  : (C,R,G,H)\label{eq:ceg2b}
\end{equation}
Now the final rule of the second derivation must apply $\ell$ and $n$
together, so our inductive job is to compose~\eqref{eq:ceg5a} with
\begin{gather}
  u:C,v:D,w:E \types (w,v,u,s \mid\,) : (E,D,C,G)\label{eq:ceg1b}
\end{gather}
The latter is obtained from the identity rule, so our task is now to
apply \cref{thm:prop-onecutadm} to the former and the single generator
$s:()\to G$.  Peeling down the derivation of the former, we obtain
\[ x:A,y:B \types (g,f\s3(y),f\s1(y),f\s2(y) \mid h(x)) : (Q,P,C,E) \]
and then
\[ x:A,y:B \types (y,x,g \mid\,) : (B,A,Q) \]
which is also obtained from the identity rule.
The identity rule can therefore also give us
\[ x:A,y:B \types (y,x,g,s \mid\,) : (B,A,Q,G). \]
Re-applying $f,h$ and then $k$, we obtain
\[ x:A,y:B \types (g,f\s3(y),f\s1(y),f\s2(y),s \mid h(x)) : (Q,P,C,E,G) \]
and then
\[ x:A,y:B \types (f\s1(y),k(g,f\s3(y)),f\s2(y),s \mid h(x)) : (C,D,E,G). \]
Permuting this, we obtain
\[ x:A,y:B \types (f\s2(y),f\s1(y),k(g,f\s3(y)),s \mid h(x)) : (E,C,D,G). \]
as the result of composing~\eqref{eq:ceg5a} and~\eqref{eq:ceg1b}.

Backing out the induction one more step, we must apply $\ell$ and $n$
to this using \cref{thm:prop-onecutadm}.  We cannot apply $\ell$
directly since its domain $E$ is not active (its term $f\s2(y)$ has
height $1$ while the maximum height is $2$).  Thus, we back up to the
main premise
\[ x:A,y:B \types (g,f\s3(y),f\s1(y),f\s2(y),s \mid h(x)) : (Q,P,C,E,G) \]
in which $E$ is active.
Thus, we can apply $\ell$ in the same generator rule as $k$, obtaining
\begin{equation}
  x:A,y:B \types (\ell\s1(f\s2(y)),\ell\s2(f\s2(y)),f\s1(y),k(g,f\s3(y)),s \mid h(x))
  : (H,R,C,D,G).\label{eq:ceg4}
\end{equation}
Now the domain $D$ of the generator $n$ \emph{is} active, so we can
directly apply it with another generator rule, obtaining (after
permutation)
\begin{equation}
  x:A,y:B \types (f\s1(y),\ell\s2(f\s2(y)),s,\ell\s1(f\s2(y))
  \mid n(k(g,f\s3(y))),h(x)) : (C,R,G,H).\label{eq:ceg3}
\end{equation}
as the result of composing~\eqref{eq:ceg5a} and~\eqref{eq:ceg2b}.

Finally, we must compose this with $m$ using
\cref{thm:prop-onecutadm}.  Neither of the domain types $C$ and $R$ is
active in~\eqref{eq:ceg3} (in fact, \emph{no} types are active
in~\eqref{eq:ceg3}, since the last rule applied was a generator rule
whose only generator has nullary codomain), so we have to inductively
peel back to~\eqref{eq:ceg4} in which $R$ is active (though not $C$).
Thus, we can then apply $m$ in the same generator rule as $n$,
obtaining
\begin{equation}
  x:A,y:B \types (m(f\s1(y),\ell\s2(f\s2(y))),s,\ell\s1(f\s2(y))
  \mid n(k(g,f\s3(y))),h(x)) : (F,G,H)\label{eq:ceg}
\end{equation}
as our end result.

Note that the terms in~\eqref{eq:ceg} are indeed the result of
substituting $f\s1(y)$ for $u$, $k(g,f\s3(y))$ for $v$, and $f\s2(y)$
for $w$ (the terms appearing in~\eqref{eq:ceg5a}) in the terms
of~\eqref{eq:ceg5b}, and appending the scalar term $h(x)$
of~\eqref{eq:ceg5a} to the scalar terms of~\eqref{eq:ceg5b}:
\begin{align*}
  m(u,\ell\s2(w))[f\s1(y)/u, k(g(f\s3(y)))/v, f\s2(y)/w]
  &= m(f\s1(y),\ell\s2(f\s2(y)))\\
  s[f\s1(y)/u, k(g(f\s3(y)))/v, f\s2(y)/w]
  &= s \\
  \ell\s1(w)[f\s1(y)/u, k(g(f\s3(y)))/v, f\s2(y)/w]
  &= \ell\s1(f\s2(y))\\
  n(v)[f\s1(y)/u, k(g(f\s3(y)))/v, f\s2(y)/w]
  &= n(k(g(f\s3(y)))).
\end{align*}
The only choice involved in the proof of \cref{thm:prop-onecutadm} is
in how to order the scalar terms in the result.  We adopted the
convention that those associated to the terms being substituted into
come first, followed by those associated to the terms being
substituted.  The opposite convention would do as well for the
following theorem:

\begin{prop}\label{thm:prop-cutassoc}
  Composition is associative and unital.
\end{prop}
\begin{proof}
  Since derivations are determined uniquely by their terms by
  \cref{thm:prop-tad}, this follows from the evident associativity and
  unitality of substitution into terms, and the associativity and
  unitality of concatenation of lists of scalar terms.
\end{proof}

Thus we have a category whose objects are contexts and whose morphisms
are derivable judgments $\Gamma\types (\vec{M}\mid\vec{Z}) : \Delta$.
However, this is not quite the underlying category of our prop: we
must quotient it by the equality rule from \cref{fig:equality}:
\begin{equation}\label{eq:prop-perm-2}
  \inferrule{\Gamma\types (\vec M \mid Z_1,\dots,Z_n) : \Delta \\
    \rho \in S_n \\
    \si : \fA \cong \fA}{\Gamma\types (\vec M \mid Z_1,\dots,Z_n) =
    (\vec M^\si \mid Z^\si_{\rho 1},\dots,Z^\si_{\rho n}) : \Delta}
\end{equation}
(which also ensures that the choice of ordering in
\cref{thm:prop-cutadm} is irrelevant).  For this we need the evident
observation:

\begin{prop}\label{thm:freeprop-cong}
  The equality rule~\eqref{eq:prop-perm-2} is a congruence on the
  category of contexts and derivable typing judgments.  That is, it is
  an equivalence relation on the morphisms that is preserved by
  composition on both sides.\qed
\end{prop}

Therefore, the quotient by this equality judgment is again a category
whose objects are the contexts (i.e.\ finite lists of types).

\begin{thm}\label{thm:prop-moncat}
  The contexts and derivable term judgments in the type theory for the
  free prop generated by \cG, modulo the equality
  rule~\eqref{eq:prop-perm-2}, form a symmetric strict monoidal
  category.
\end{thm}
\begin{proof}
  The monoidal structure on contexts is concatenation, with the empty
  context as unit.  To tensor morphisms, it is easiest to first tensor
  with identities: given $\Gamma\types (\vec M\mid \vec Z):\Delta$, we
  construct
  $\Gamma,\vec x:\vec A\types (\vec M,\vec x\mid \vec Z):\Delta,\vec
  A$ by inducting until we get down to the identity rule and then
  adding the variables $\vec{x}:\vec{A}$ to the context.  Now we
  obtain the tensor product of
  $\Gamma\types (\vec{M}\mid\vec{Y}) : \Delta$ and
  $\Phi\types (\vec{N}\mid \vec{Z}): \Psi$ by first tensoring with
  identities to get
  $\Gamma,\Phi \types (\vec{M},\Gamma \mid\vec{Y}): \Delta,\Phi$ and
  $\Delta,\Phi\types (\Delta,\vec{N}\mid\vec{Z}) : \Delta,\Psi$ and
  then composing to get
  $\Gamma,\Phi\types (\vec{M},\vec{N}\mid \vec{Z},\vec{Y}) :
  \Delta,\Psi$.  If we did this in the other order, we would get
  $(\vec{M},\vec{N}\mid \vec{Y},\vec{Z})$ instead, which is equal
  by~\eqref{eq:prop-perm-2}.  In particular, this implies
  functoriality of the tensor product; associativity and unitality
  follow similarly.  Finally, the symmetry isomorphism is
  $\vec x:\vec A, \vec y:\vec B \types (\vec y,\vec x) : \vec B,\vec
  A$; it is easy to verify the axioms.
\end{proof}

Thus we have a prop, which we denote $\F{}\cG$.

\begin{thm}\label{thm:prop-initial}
  $\F{}\cG$ is the free prop generated by \cG.
\end{thm}
\begin{proof}
  Let \cP be a prop and $\omega:\cG\to \cP$ a morphism of signatures.
  We extend it to a morphism of props $\F{}\cG \to \cP$ by induction
  on derivations.  By the coherence theorem for symmetric monoidal
  categories (e.g.~\cite[Chapter XI]{maclane}), there is a unique
  choice at each step if we are to have a (symmetric strict monoidal)
  functor, and likewise both equality rules corresponds to actual
  equalities that must hold in \cP.  Then we prove that this actually
  is a symmetric strict monoidal functor, using the definition of
  composition and the tensor product in $\F{}\cG$.
\end{proof}

\begin{rmk}
  Since the free prop generated by a signature is unique up to
  isomorphism, it follows that our $\F{}\cG$ is isomorphic to the free
  prop on \cG presented in any other way, such as by using string
  diagrams whose edges and vertices are labeled by objects and
  morphisms of \cG respectively.  Unsurprisingly, this correspondence
  can be made more explicit: from any derivable
  $\Gamma \types \vec M : \Delta$ we can construct a labeled string
  diagrams from $\Gamma$ to $\Delta$, whose vertices are the generator
  applications and whose edges are the disjoint union of the variables
  in $\Gamma$ and the terms $f\s i(\vec N)$ appearing as subterms of
  $\vec M$.
\end{rmk}

\section{Presentations of props}
\label{sec:presentations-props}

Since the category of props is monadic over the category of signatures
by \cref{thm:props-monadic}, every prop \cP has a \emph{presentation}
in terms of signatures, i.e.\ a coequalizer diagram
\[ \F{}\cR \toto \F{}\cG \to \cP.
\]
Moreover, since $\F{}$ and its right adjoint are the identity on the
set of objects, we may assume that $\cR$ and $\cG$ both have the same
set of objects as \cP and all the morphisms are the identity on
objects.  Now by the universal property of $\F{}\cR$, the two
morphisms of props $\F{}\cR \toto \F{}\cG$ are equivalently morphisms
of signatures $\cR \toto \F{}\cG$.  Thus, once $\cG$ is given, the
additional data of $\cR$ consists of a set of pairs of parallel
morphisms in $\F{}\cG$, which is to say a set of \textbf{equality
  axioms} of the form
\[ \Gamma \types (\vec{M}\mid\vec{Y}) = (\vec{N} \mid\vec{Z}) : \Delta
\]
where both $\Gamma \types (\vec{M}\mid\vec{Y}) : \Delta$ and
$\Gamma\types (\vec{N} \mid\vec{Z}):\Delta$ are derivable in the type
theory for the free prop generated by \cG.  We obtain the \textbf{type
  theory for the prop presented by $(\cG,\cR)$} by augmenting the type
theory for the free prop on \cG by these axioms for the equality
judgment, together with the additional rules shown in
\cref{fig:equality2} (which are no longer automatic in the presence of
such axioms).

\begin{figure}
  \centering
  \begin{mathpar}
    \inferrule{\Gamma\types (\vec M\mid\vec Z) : \Delta}
    {\Gamma\types (\vec M\mid\vec Z) = (\vec M\mid\vec Z) : \Delta}
    \and
    \inferrule{\Gamma\types (\vec M\mid\vec Y) = (\vec N \mid \vec Z) : \Delta}
    {\Gamma\types (\vec N \mid \vec Z) = (\vec M\mid\vec Y) : \Delta}
    \and
    \inferrule{\Gamma\types (\vec M\mid\vec X) = (\vec N \mid \vec Y) : \Delta \\
      \Gamma\types (\vec N \mid \vec Y)  = (\vec P\mid\vec Z) : \Delta}
    {\Gamma\types (\vec M\mid\vec X) = (\vec P \mid \vec Z) : \Delta}
    \and
    \inferrule{\Gamma\types (\vec M\mid\vec X) = (\vec N \mid \vec Y) : \Delta \\
      \Delta \types (\vec P \mid \vec Z) : \Phi}
    {\Gamma\types
      (\vec P[\vec M/\Delta] \mid \vec Z[\vec M/\Delta], \vec X) =
      (\vec P[\vec N/\Delta] \mid \vec Z[\vec N/\Delta], \vec Y) : \Phi}
    \and
    \inferrule{\Gamma\types (\vec M\mid\vec X) : \Delta \\
      \Delta \types (\vec N \mid \vec Y) = (\vec P \mid \vec Z) : \Phi}
    {\Gamma\types
      (\vec N[\vec M/\Delta] \mid \vec Y[\vec M/\Delta], \vec X) =
      (\vec P[\vec M/\Delta] \mid \vec Z[\vec M/\Delta], \vec X) : \Phi}
    \and
    \inferrule{\Gamma \types (\vec M\mid\vec X) = (\vec N \mid \vec Y) : \Delta \\
      \Phi \types (\vec P\mid\vec Z) = (\vec Q \mid \vec W) : \Psi}
    {\Gamma,\Phi \types (\vec M,\vec P \mid \vec X, \vec Z) =
      (\vec N, \vec Q \mid \vec Y,\vec W) : \Delta,\Psi}
    \and
    \inferrule{\Gamma\types (\vec M \mid Z_1,\dots,Z_n) : \Delta \\
      \rho \in S_n \\
      \si : \fA \cong \fA}{\Gamma\types (\vec M \mid Z_1,\dots,Z_n) =
      (\vec M^\si \mid Z^\si_{\rho 1},\dots,Z^\si_{\rho n}) : \Delta}
  \end{mathpar}
  \caption{Rules of the equality judgment in the presence of axioms}
  \label{fig:equality2}
\end{figure}

The first three rules in \cref{fig:equality2} are the usual
reflexivity,\footnote{Actually, it is not necessary to assert
  reflexivity explicitly, since it is a special case of the
  permutation rule.} symmetry, and transitivity.  The next three are
congruence rules for precomposition, postcomposition, and the
concatenation product.  The final one is the label-renaming and
scalar-permutation rule from \cref{fig:equality}.  Note that
congruence for pre- and post-composition includes congruence under
permutations of the domain and codomain.

\begin{prop}
  For any prop presentation $(\cG,\cR)$, the equality judgment
  generated by the axioms of \cR together with the rules of
  \cref{fig:equality2} is a congruence on the prop $\F{}\cG$.  That
  is, it is an equivalence relation on morphisms preserved by
  composition and tensor product.
\end{prop}
\begin{proof}
  The rules of \cref{fig:equality2} essentially force this to be true.
\end{proof}

Thus we obtain a prop $\F{}\langle\cG|\cR\rangle$ as the quotient of
$\F{}\cG$ by this congruence.

\begin{thm}
  $\F{}\langle\cG|\cR\rangle$ is the prop presented by $(\cG,\cR)$.
  That is, we have a coequalizer diagram in the category of props:
  \[ \F{}\cR \toto \F{}\cG \to \F{}\langle\cG|\cR\rangle. \]
\end{thm}
\begin{proof}
  Since the quotient map $\F{}\cG \to \F{}\langle\cG|\cR\rangle$ is
  surjective, a prop morphism $\omega:\F{}\cG \to \cP$ factors through
  $\F{}\langle\cG|\cR\rangle$ in at most one way.  Moreover, it does
  so precisely when it identifies all pairs of tuples of terms that
  are identified by the equality judgment.  However, the equality
  rules in \cref{fig:equality2} are satisfied in any prop, so this
  happens precisely when $\omega$ respects the axioms of \cR, i.e.\
  when it coequalizes the two maps $\F{}\cR \toto \F{}\cG$.
\end{proof}

Thus, we can use our type theory to reason about structures in
arbitrary props defined by generators and axioms, and hence also in
symmetric monoidal categories (by \cref{thm:smc-prop}).

\section{Examples}
\label{sec:examples}

\subsection{Duals and traces}
\label{sec:duals-traces}

We begin by repeating the first example from the introduction more
carefully.  The \textbf{free prop generated by a dual pair} has a
generating signature $\cG$ with two objects $A$ and $A^*$ and two
morphisms $\eta:()\to (A,A^*)$ and $\ep:(A^*,A)\to ()$, and a
signature of relations \cR that imposes two axioms
\begin{mathpar}
  x:A \types (\eta\s1 \mid \ep(\eta\s2,x)) = x :A\and
  y:A^* \types (\eta\s2 \mid \ep(y,\eta\s1)) = y :A^*.
\end{mathpar}
A map from this prop to a symmetric monoidal category then reduces to
the usual notion of dual.

As suggested in the introduction, we write $\eval{M}{N}$ for
$\ep(M,N)$, and $(u,\abs{A} u) : (A,A^*)$ for $(\eta\s1,\eta\s2)$,
where $u$ is a label (i.e.\ an element of $\fA$) rather than a
variable (appearing in the context).  In this notation, the axioms are
\begin{mathpar}
  x:A \types (u \mid \eval{\abs{A}u}{x}) = x :A\and
  y:A^* \types (\abs{A} u \mid \eval{y}{u}) = y :A^*.
\end{mathpar}
Recall that $=$ is a congruence for substitution on both sides.  Thus
the first axiom means that if $\eval{\abs{A}u}{M}$ appears in the
scalars, for \emph{any term} $M:A$ not involving $u$, then it can be
removed by replacing $u$ (wherever it appears, even as a subterm of
some other term) with $M$.  This justifies regarding it as a sort of
``$\beta$-reduction for duality'' with $u$ playing the role of a
``bound variable'', although as we noted in the introduction the
``binder'' $\abs{A} u$ does not delimit the ``scope'' of $u$ at all.
Running this rule in reverse, we see that {any term} $M:A$ (appearing
even as a sub-term of some other term) can be replaced by $u$, for a
fresh label $u$, if we simultaneously add $\eval{\abs{A}u}{M}$ to the
scalars.  The other axiom is similar; we may regard it as an
``$\eta$-reduction for duality''.

If $A$ has a dual $A^*$, and $f:A\to A$, the \textbf{trace} of $f$ is the composite
\[ () \xto{\eta} (A,A^*) \xto{(f,\idfunc)} (A,A^*) \xto{\cong} (A^*,A)
  \xto{\ep} () \] In our type theory this is
\[ () \types (\,\mid \eval{\abs{A} u}{f(u)}) : ().
\]
As advertised in the introduction, we can now prove cyclicity of the
trace with a $\beta$-expansion followed by a $\beta$-reduction: for
morphisms $f:A\to B$ and $g:B\to A$, with $A$ and $B$ dualizable, we
have
\begin{align*}
\tr(f g) &\defeq (\,\mid \eval{\abs{B} y}{f(g(y))})\\
&= (\,\mid \eval{\abs{B}{y}}{f(x)},\eval{\abs{A}{x}}{g(y)})\\
&= (\,\mid \eval{\abs{A} x}{g(f(x))})\\
&\defeq \tr(g f).
\end{align*}

More generally, any $f:(Y,A) \to (Z,A)$ has a ``partial'' or ``twisted'' trace
\[ y:Y \types (f\s1(y,u) \mid \eval{\abs{A}u }{f\s2(y,u)} ) : Z.
\]
This satisfies a version of cyclicity~\cite[Lemma 4.4]{ps:symtraces}:
for any morphisms $f:(Y,A) \to (Z,B)$ and $g:(W,B) \to (X,A)$, with
$A$ and $B$ dualizable, we have
\begin{align*}
  y:Y, w:W &\types (g\s1(w,v), f\s1(y,g\s2(w,v)) \mid \eval{\abs{B} v}{f\s2(y,g\s2(w,v))}) \\
  &= (g\s1(w,v), f\s1(y,u) \mid \eval{\abs{B} v}{f\s2(y,u)}, \eval{\abs{A} u}{g\s2(w,v)})\\
  &= (g\s1(w,f\s2(y,u)), f\s1(y,u) \mid \eval{\abs{A} u}{g\s2(w,f\s2(y,u))})\\
  &: (X,Z).
\end{align*}

This general cyclicity includes in particular the \emph{sliding} axiom
for traces from~\cite{jsv:traced-moncat}.  Their other axioms become
simply syntactic identities in our type theory.  For instance,
\emph{tightening} is the statement that if we compose
$f:(Y,A) \to (Z,A)$ with $u:X\to Y$ and $v:Z\to W$ and then take its
trace:
\begin{mathpar}
  \small
  \inferrule*{
    \inferrule*{x:X \types u(x):Y \\
      y:Y, a:A \types (f\s1(y,a),f\s2(y,a)): (Z,A) \\ z:Z \types v(z):W}
    {x:X, a:A \types (v(f\s1(u(x),a)), f\s2(u(x),a)) : (W,A)}}
  {x:A \types (v(f\s1(u(x),a)) \mid \eval{\abs{A} a}{f\s2(u(x),a)}) : W}
\end{mathpar}
we get the same result as if we first take the trace of $f$ and then
compose with $u$ and $v$:
\begin{mathpar}
  \small
  \inferrule*{x:X \types u(x):Y \\
    \inferrule*{y:Y, a:A \types (f\s1(y,a),f\s2(y,a)): (Z,A)}
    {y:Y \types (f\s1(y,a) \mid \eval{\abs{A} a}{f\s2(y,a)}) : Z} \\
    z:Z \types v(z):W}
  {x:A \types (v(f\s1(u(x),a)) \mid \eval{\abs{A} a}{f\s2(u(x),a)}) : W.}
\end{mathpar}
Since the terms concluding both derivations are the same, they
represent the same morphism.  Of course, these are not actually
derivations in our type theory, since they use the admissible rule of
substitution.  The uniqueness of typing derivations means that if the
substitutions are eliminated according to \cref{thm:prop-cutadm} we
obtain the same result in both cases, which in this case is:
\begin{mathpar}
  \small
  \inferrule*{\inferrule*{\inferrule*{ }
      {x:A \types (u(x),a,\abs{A}a) : (Y,A,A^*)}}
    {x:A \types (f\s1(u(x),a), \abs{A} a, f\s2(u(x),a)) : (Z,A^*,A)}}
  {x:A \types (v(f\s1(u(x),a)) \mid \eval{\abs{A} a}{f\s2(u(x),a)}) : W.}
\end{mathpar}

\begin{rmk}
  Given any signature \cG, we can augment it by adding a specified
  dual $A^*$ for each object $A\in\cG$ along with duality data as
  above.  This yields a new signature $\cG^*$ such that $\F{}\cG^*$ is
  freely generated by \cG together with a specified dual for each of
  its object.  Thus $\F{}\cG^*$ is the free \textbf{compact closed
    prop} (i.e.\ prop in which every object has a dual) generated by
  \cG.  Any signature \cR of relations for \cG carries over to $\cG^*$
  as well, so we can construct presented compact closed props as well,
  and thereby presented compact closed monoidal categories.

  In general, a prop presentation may not have any decision procedure
  for equality or normal forms for morphisms.  However, the view of
  the zigzag equalities as ``$\beta$ and $\eta$ reductions'' suggests
  that this should be the case for some classes of presented compact
  closed props (whenever the equalities other than the zigzag
  identities can be controlled).  If true, this should include in
  particular the explicit description of the free compact closed
  monoidal category on an ordinary category from~\cite{kl:cpt}.

  Note also that by the ``Int-construction''~\cite{jsv:traced-moncat},
  any \emph{traced} symmetric monoidal category embeds
  fully-faithfully and trace-preservingly in a compact closed one.
  Thus, we can also use presented compact closed props to reason about
  traced symmetric monoidal categories.
\end{rmk}

\subsection{Well-idempotent dualizable objects are self-dual}
\label{sec:idem-selfdual}

Tim Campion asked in~\cite{campion:idem-selfdual} whether an
idempotent dualizable object in a symmetric monoidal category must be
self-dual.  A proof using string diagrams that this holds assuming
``well-idempotence'' was given by the user ``MTyson''; here we recast
this proof in our type theory.

We assume given one type $X$ with a dual $X^*$, expressed as before,
and a morphism $i : () \to X$ such that
$1_X \otimes i : X\to X\otimes X$ is an isomorphism (such an $i$ is
what makes $X$ \emph{well-idempotent}).  In our type theory, the
latter can be expressed by a term $\types i:X$ and a morphism
$x:X,y:X \types f(x,y):X$ (the inverse to $1_X \otimes i$) such that
$f(x,i)= x$ and $(x,y) = (f(x,y),i)$.  More precisely, the latter two
equalities are
\begin{align*}
  x:X &\types f(x,i) = x : X \\
  x:X, y:X &\types (x,y) = (f(x,y),i) : (X,X)
\end{align*}
but we tend to omit the contexts and types in equality axioms and
calculations when they are obvious.

Note that the equation $f(x,i)=x$ means that $i$ is a ``right unit''
for the ``binary operation'' $f$.  We now observe that it is also a
left unit:
\begin{align*}
  f(i,y)
  &= (f(u,y) \mid \eval{\abs{X} u}{i})\\
  &= (u \mid \eval{\abs{X} u}{y})\\
  &= y.
\end{align*}
Here the first line is a $\beta$-expansion and the third is a
$\beta$-reduction.  The second line uses the equality
$(x,y) = (f(x,y),i)$ in the following way: first we introduce an extra
variable to get
\[ x:X, w:X^*, y:X \types (x,w,y) = (f(x,y),w,i) : (X,X^*,X) \]
then we precompose with
\[ y:X \types (u,\abs{X}u,y) : (X,X^*,X) \]
to get
\[ y:X \types (u, \abs{X}u,y) = (f(u,y),\abs{X}u,i) : (X,X^*,X) \]
and then we post-compose with
\[ x:X, w:X^*, y:X \types (x \mid \eval{w}{y}) : X \]
to get
\[ y:X \types (u \mid \eval{\abs{X} u}{y}) = (f(u,y) \mid \eval{\abs{X} u}{i}) : X.
\]
Note that although the type-theoretic justification is a bit
complicated, at the level of terms the operation is intuitive: we
simply simultaneously substitute $u$ for $f(u,y)$ and $y$ for $i$,
wherever the latter appear as subterms.  From now on we will perform
such substitutions at term-level without further comment.

Now we define $x:X \types \phi(x) :X^*$ and $w:X^* \types \psi(w):X$ by
\begin{align*}
  \phi(x) &\defeq (\abs{X} v \mid \eval{\abs{X} u}{f(v,f(x,u))})\\
  \psi(w) &\defeq (i \mid \eval w i).
\end{align*}
Finally, we can show that $\phi$ and $\psi$ are inverse isomorphisms
(so that $X$ is isomorphic to its dual) with the following
computations:
\begin{alignat*}{2}
  \psi(\phi(x))
  &= (i \mid \eval{\abs{X} v}{i}, \eval{\abs{X} u}{f(v,f(x,u))})\\
  &= (i \mid \eval{\abs{X} u}{f(i,f(x,u))}) &\quad (\beta\text{-reduction for }v)\\
  &= (i \mid \eval{\abs{X} u}{f(x,u)}) &\quad (i\text{ is a left unit for }f)\\
  &= (u \mid \eval{\abs{X} u}{x}) &\quad ((x,u)=(f(x,u),i))\\
  &= x &\quad (\beta\text{-reduction for }u)\\
  \phi(\psi(w))
  &= (\abs{X} v \mid \eval{\abs{X} u}{f(v,f(i,u))}, \eval w i)\\
  &= (\abs{X} v \mid \eval{\abs{X} u}{f(v,u)}, \eval w i) &\quad (i\text{ is a left unit for }f)\\
  &= (\abs{X} v \mid \eval{\abs{X} u}{v}, \eval w u) &\quad ((v,u) = (f(v,u),i))\\
  &= (\abs{X} v \mid \eval w v) &\quad (\beta\text{-reduction for }u)\\
  &= w &\quad (\eta\text{-reduction for }v)\mathrlap{.}
\end{alignat*}
We do not reproduce MTyson's string diagram proof here, but the reader
is encouraged to compare it to our type-theoretic version.  Note that
this situation is partly ``topological'' in the sense of
\cref{sec:vistas} (the zigzag axioms for duality, and arguably the
unit properties $f(x,i)=x$ and $f(i,x)=x$) and partly non-topological
(the axiom $(x,y) = (f(x,y),i)$).

\subsection{Comonoids and Sweedler notation}
\label{sec:sweedler-notation}

As in ordinary cartesian (or even linear) type theory, it is easy to
use our type theory to define \textbf{monoid objects} in a symmetric
monoidal category.  The signature has one object $M$ and two morphisms
$m:(M,M) \to M$ and $e:() \to M$; we usually write $m(x,y)$ infix as
$x\cdot y$ or just $x y$.  The axioms are the expected
\begin{mathpar}
  (xy)z = x(yz) \and xe=x \and ex =x.
\end{mathpar}
However, with our type theory we can also define \textbf{comonoid
  objects}, which have instead two morphisms $\comult : M\to (M,M)$
and $\ep : M \to ()$, and axioms
\begin{mathpar}
  (\comult\s1(\comult\s1(x)),\comult\s2(\comult\s1(x)),\comult\s2(x))
  =(\comult\s1(x),\comult\s1(\comult\s2(x)),\comult\s2(\comult\s2(x)))
  \and
  (\comult\s1(x)\mid\ep(\comult\s2(x)))
  \and
  (\comult\s2(x)\mid\ep(\comult\s1(x))).
\end{mathpar}

As suggested in the introduction, this becomes much more manageable if
we adopt the convention of traditional \emph{Sweedler notation} for
comonoids and comodules:
\[ (x\s1,x\s2) \defeq (\comult\s1(x),\comult\s2(x)) .
\]
Since there is no other meaning of $Z\s i$ when $Z$ is itself already
a term, it is unambiguous to regard it as meaning $\comult\s i(Z)$ as
long as no type has more than one relevant comultiplication.  (We
could formalize this with a more complicated type-checking algorithm,
but we will be content to regard it as an informal abuse of notation.)
We may regard this as a sort of ``dual'' to the shorthand notation
``$x y$'' for multiplication in a monoid, which omits the name or
symbol for the product $(M,M)\to M$.  With this notation, the axioms
of a comonoid become
\begin{mathpar}
  (x\ss11,x\ss12,x\s2) = (x\s1,x\ss21,x\ss22)
  \\
  (x\s1 \mid \ep(x\s2)) = x
  \and
  (x\s2 \mid \ep(x\s1)) = x.
\end{mathpar}
Traditional Sweedler notation also goes one step further: in view of
the coassociativity axiom, it is unambiguous to write
$(x\s1,x\s2,x\s3)$ for either $(x\ss11,x\ss12,x\s2)$ or
$(x\s1,x\ss21,x\ss22)$.  In general, if subscripts are applied to a
variable or a term that is already
subscripted,
with the maximum such subscript being $n$, then the subscript ${}\s k$
means $\overbrace{{}\s2{}\s2\dots{}\s2}^{k-1}{}\s1$ if $k<n$ and
$\overbrace{{}\s2{}\s2\dots{}\s2}^{n-1}$ if $k=n$.

Intuitively, in cartesian type theory (i.e.\ in a cartesian monoidal
category), everything can be duplicated and discarded with impunity;
whereas a comonoid in a non-cartesian monoidal category is equipped
with \emph{specified ways} in which to duplicate and discard elements.
(Indeed, a cartesian monoidal category is precisely a symmetric
monoidal category in which every object is equipped with a commutative
comonoid structure in a natural way.)  We thus view $x\s1$ and $x\s2$
as ``duplicated copies'' of $x$, the subscripts tracking the order of
duplication.  Similarly, we can regard $\ep(x)$ as ``discarding'' the
element $x$, which inspires us to introduce a sort of ``nullary
Sweedler notation''
\[ \cancel{x} \defeq \ep(x).
\]
Thus, for instance, the counit axioms of a comonoid become
$(x\s1\mid \cancel{x\s2})=x$ and $(x\s2\mid \cancel{x\s1})=x$.  Note
that by coassociativity, we also have equalities such as
$(x\s1, x\s2 \mid\cancel{x\s3}) = (x\s1,x\s2)$ and
$(x\s1, x\s3 \mid\cancel{x\s2}) = (x\s1,x\s2)$ and so on.

Such shorthands need not be restricted to comonoids either.  For
instance, traditional Sweedler notation is also used for
\emph{comodules}, which have a coaction $D \to (C,D)$ by a coalgebra
$C$.  As an example with even greater generality, suppose $M$ is
dualizable and has a coaction $\comult:M\to (A,M)$, satisfying no
axioms at all, and that $f:M\to M$ is an endomorphism that respects
$\comult$ in that $(f(x\s1),f(x\s2)) = (f(x)\s1,f(x)\s2)$.  Then we
can use this notation to verify the \emph{fixed-point property} of
traces from~\cite[Corollary 5.3]{ps:symtraces}:
\begin{align*}
  (f(f(u)\s1) \mid \eval{\abs{M} u}{f(u)\s2})
  &= (f(v\s1) \mid \eval{\abs{M} u}{v\s2}, \eval{\abs{M} v}{f(u)})\\
  &= (f(v\s1) \mid \eval{\abs{M} v}{f(v\s2)})\\
  &= (f(v)\s1 \mid \eval{\abs{M} v}{f(v)\s2}) : A.
\end{align*}

\subsection{Frobenius monoids}
\label{sec:frobenius}

A \textbf{Frobenius monoid} is an object that is both a monoid and a
comonoid and satisfies the additional axiom
\begin{equation}
  x:M, y:M \types (x\s1, x\s2 y) = (x y\s1, y\s2).\label{eq:frob}
\end{equation}
Usually this is stated as two axioms saying that both sides of the
above equation equal $((xy)\s1, (xy)\s2)$.  But this follows from the
above axiom by the following argument, which I learned
from~\cite{ps:weak-hopf}:
\begin{alignat*}{2}
  (x\s1, x\s2 y)
  &= (x\s1, x\s2 y\s1 \mid \cancel{y\s2}) &&\qquad \text{(by counitality)}\\
  &= (x\s1, x\ss21 \mid \cancel{x\ss22 y}) &&\qquad \text{(by~\eqref{eq:frob})}\\
  &= (x\ss11, x\ss12 \mid \cancel{x\s2 y}) &&\qquad \text{(by coassociativity)}\\
  &= ((x y\s1)\s1, (x y\s1)\s2 \mid \cancel{y\s2}) &&\qquad \text{(by~\eqref{eq:frob})}\\
  &= ((x y)\s1, (x y)\s2) &&\qquad\text{(by counitality)}\mathrlap{.}
\end{alignat*}
The Frobenius axiom(s) are ``topological'', so their string diagrams
get a good deal of leverage from topological intuition.  Thus,
Frobenius monoids are not a very good example for the relative
usefulness of type theory.  However, for purposes of comparison, we
include a proof of one of the basic facts about Frobenius monoids;
namely that they are self-dual, with unit and counit:
\begin{align*}
  (\eval w x) &\defeq (\,\mid \cancel{w x})\\
  (u, \abs{M} u) &\defeq (e\s1, e\s2).
\end{align*}
The axioms of a dual pair follow quite easily:
\begin{mathpar}
  (u \mid \eval{\abs{M}u}{x})
  \defeq (e\s1 \mid \cancel{e\s2 x})
  = (e x\s1 \mid \cancel{x\s2})
  = ex
  = x\\
  (\abs{M}u \mid \eval w u)
  \defeq (e\s2 \mid \cancel{w e\s1})
  = (w\s2 e \mid \cancel{w\s1})
  = w e
  = w.
\end{mathpar}

\begin{rmk}
  A \textbf{hypergraph category}~\cite{kissinger:mat-hycat} is a
  symmetric monoidal category in which every object is equipped with a
  Frobenius monoid structure that is commutative ($xy = yx$),
  cocommutative ($(x\s1,x\s2)=(x\s2,x\s1)$), and special a.k.a.\
  separable ($x\s1\,x\s2 = x$), and such that the Frobenius monoid
  structure on any tensor product $X\otimes Y$ is induced from those
  on $X$ and $Y$ in the standard way.  The definition of
  \textbf{hypergraph prop} is a bit simpler: it is just a prop in
  which every object is equipped with a commutative, cocommutative,
  special Frobenius monoid structure.  Since tensor products in a prop
  are only formal, the final condition is essentially automatic.  More
  specifically, the final condition on a hypergraph category \cC is
  not needed to show that it has an underlying hypergraph prop $U\cC$,
  but it is precisely what is needed to show that \cC is equivalent,
  as a hypergraph category, to the free symmetric monoidal category
  generated by $U\cC$.  (One might even argue that for this reason,
  hypergraph props are a more natural structure than hypergraph
  categories.)

  Now given any signature \cG, we can augment it by adding a
  commutative, cocommutative, special Frobenius monoid structure on
  every object.  This yields a new signature $\cG^\hy$ such that
  $\F{}\cG^\hy$ is the free hypergraph prop generated by \cG.  Any
  signature \cR of relations for \cG carries over to $\cG^\hy$ as
  well, so we can construct presented hypergraph props as well.
\end{rmk}

\subsection{Hopf monoids and antipodes}
\label{sec:antipodes}

A monoid object in a cartesian monoidal category is also (like every
object) a comonoid, but the monoid and comonoid structures do not
satisfy the Frobenius axiom.  Instead they satisfy the
\textbf{bimonoid axioms}:
\begin{align*}
  x:M,y:M &\types (x\s1 y\s1,x\s2 y\s2) = ((x y)\s1,(x y)\s2) :(M,M)\\
          &\types (e\s1,e\s2)=(e,e):(M,M)\\
  x:M,y:M &\types (\mid\cancel{x y}) = (\mid\cancel{x},\cancel{y}) : ()\\
  &\types (\mid\cancel{e})=():().
\end{align*}
Thus, a bimonoid object in an arbitrary monoidal category can be
regarded as a ``non-cartesian'' version of a monoid object in a
cartesian monoidal category.  Indeed, if a bimonoid is cocommutative
$(x\s1,x\s2) = (x\s2,x\s1)$ then it is a monoid object in the
cartesian monoidal category of cocommutative comonoids, and dually if
it is commutative ($xy=yx$) then it is a monoid object in the opposite
of the cocartesian monoidal category of commutative monoids.  But in
the non-commutative, non-cocommutative case we obtain something truly
new.

The analogue for a bimonoid of the inversion operation, making a
monoid into a group, is called an \textbf{antipode}: an operation
$x:M \types \inv{x}:M$ such that
\begin{align*}
  x:M &\types x\s1\, \inv{x\s2} = (e\mid\cancel{x}) :M\\
  x:M &\types \inv{x\s1}\, x\s2 = (e\mid\cancel{x}) :M.
\end{align*}
A bimonoid equipped with an antipode (a non-cartesian analogue of a
group object) is called a \textbf{Hopf monoid}.
Note that the comonoid structure is necessary in order to even
formulate the antipode axioms: we need to duplicate $x$ in order to
invert one copy of it, and on the other side of the equation we need
to discard $x$ in order to write simply ``$e$''.  In a cartesian
monoidal category, Hopf monoids are precisely group objects in the
usual sense.  However, note also that bimonoids and Hopf monoids in a
symmetric monoidal category are self-dual: such a structure on
$M\in\cC$ is equivalent to such a structure on $M\in\cC\op$.

As mentioned in the introduction, cartesian type theory can
internalize the basic fact of group theory that inverses in any monoid
are unique: if $\inv{x}$ and $\invt{x}$ are both inverses of $x$ the
\[ \inv{x} = \inv{x}e = \inv{x}(x\invt{x}) =
  (\inv{x}x)\invt{x} = e\invt{x} = \invt{x}.
\]
Therefore, a monoid object in any cartesian monoidal category admits
at most one inverse, and hence both \emph{cocommutative} Hopf monoids
and \emph{commutative} ones have unique antipodes.  Cartesian type
theory has nothing to say about Hopf monoids that are neither
commutative nor cocommutative, but in our type theory we can reproduce
essentially the same argument: if $x:M\types \inv{x}:M$ and
$x:M \types \invt{x}:M$ are both antipodes, we compute
\begin{align*}
  \inv{x}
  = \inv{x}\, e
  = (\inv{x\s1}\, e\mid\cancel{x\s2})
  = \inv{x\s1}\, x\s{21} \, \invt{x\s{22}}
  = (e \, \invt{x\s{2}}\mid\cancel{x\s1})
  = e\, \invt{x}
  = \invt{x}.
\end{align*}
Thus, even in a non-cartesian situation we can use a very similar
set-like argument, as long as we keep track of where elements get
``duplicated and discarded''.  I encourage the reader to write out a
proof of this fact using traditional arrow notation or string diagrams
for comparison.

As another example, if $H$ and $K$ are bimonoids, a \textbf{bimonoid
  homomorphism} is a morphism $x:H \types f(x) :K$ such that
\begin{mathpar}
  f(x y) = f(x)\, f(y)\and
  (f(x)\s1,f(x)\s2) = (f(x\s1), f(x\s2))\\
  f(e) = e \and
  (\mid \cancel{f(x)}) = (\mid \cancel{x}).
\end{mathpar}
Now we can show that when $H$ and $K$ are Hopf monoids, any such
bimonoid homomorphism preserves antipodes.
\begin{align*}
  f(\inv{x})
  &= (f(\inv{x\s1}) \mid \cancel{x\s2})\\
  &= (f(\inv{x\s1}) \mid \cancel{f(x\s2)})\\
  &= (f(\inv{x\s1})\, e \mid \cancel{f(x\s2)})\\
  &= f(\inv{x\s1})\, f(x\s2)\s1 \, \inv{f(x\s2)\s2}\\
  &= f(\inv{x\s1})\, f(x\ss21) \, \inv{f(x\ss22)}\\
  &= f(\inv{x\ss11})\, f(x\ss12) \, \inv{f(x\s2)}\\
  &= f(\inv{x\ss11}\, x\ss12) \, \inv{f(x\s2)}\\
  &= (f(e) \, \inv{f(x\s2)} \mid \cancel{x\s1})\\
  &= (e \, \inv{f(x\s2)} \mid \cancel{x\s1})\\
  &= (\inv{f(x\s2)} \mid \cancel{x\s1})\\
  &= \inv{f(x)}.
\end{align*}
There is a more general approach to results of this sort, which we
will return to in the next section.

\subsection{Weak bimonoids}
\label{sec:weak-bimonoids}

A \textbf{weak
  bimonoid}~\cite{ps:weak-hopf}\footnote{In~\cite{ps:weak-hopf} these
  definitions are given in the additional generality of a
  \emph{braided} monoidal category, but our type theory only applies
  to symmetric monoidal categories.  This is sufficient to include a
  number of examples, however, such as the category algebra of a
  category with finitely many objects.} is a monoid and comonoid that
satisfies, instead of the bimonoid axioms, the following weakened
ones:
\begin{align*}
  ((xy)\s1, (xy)\s2) &= (x\s1 y\s1, x\s2 y\s2)\\
  (\,\mid\cancel{xyz}) &= (\,\mid \cancel{x y\s1}, \cancel{y\s2 z})\\
  &= (\,\mid \cancel{x y\s2}, \cancel{y\s1 z})\\
  (e\s1, e\s2, e\s3) &= (e\s1, e\s2 e'\s1, e'\s2)\\
  &= (e\s1, e'\s1 e\s2 , e'\s2).
\end{align*}
For a weak bimonoid we define
\begin{align*}
  s(x) &\defeq (e\s1 \mid \cancel{e\s2 x})\\
  t(x) &\defeq (e\s1 \mid \cancel{x e\s2})\\
  r(x) &\defeq (e\s2 \mid \cancel{e\s1 x}).
\end{align*}

Many equations relating the weak bimonoid structure and the operations
$s,t,r$ are proven in~\cite{ps:weak-hopf} using string diagrams.  All
of them can also be proven in our type theory.  For instance, here is
a version of~\cite[Appendix B, eq.~(1)]{ps:weak-hopf}:
\begin{align*}
  (s(x)\s1, s(x)\s2)
  &= (e\ss11, e\ss12 \mid \cancel{e\s2 x})\\
  &= (e\s1, e\s2 \mid \cancel{e\s3 x})\\
  &= (e\s1, e\s2 e'\s1 \mid \cancel{e'\s2 x})\\
  &= (e\s1, e\s2 s(x)).
\end{align*}
Here is a version of~\cite[Appendix B, eq.~(4)]{ps:weak-hopf}:
\begin{alignat*}{2}
  ((x\, s(y))\s1,(x\, s(y))\s2)
  &= (x\s1 s(y)\s1, x\s2 s(y)\s2)\\
  &= (x\s1 e\s1, x\s2 e\s2 s(y)) &\quad \text{(using (1))}\\
  &= ((xe)\s1, (xe)\s2 s(y))\\
  &= (x\s1, x\s2 s(y)).
\end{alignat*}
And here is a version of~\cite[Appendix B, eq.~(13)]{ps:weak-hopf}.
\begin{align*}
  t(x)\, r(y)
  &= (e\s1 e'\s2 \mid \cancel{x e\s2}, \cancel{e'\s1 y})\\
  &= (e\s2 \mid \cancel{xe\s3}, \cancel{e\s1 y})\\
  &= (e\s2 e'\s1 \mid \cancel{x e'\s2}, \cancel{e\s1 y})\\
  &= r(y) \, t(x).
\end{align*}

As an example of a somewhat longer proof, here is a version
of~\cite[Appendix B, eq.~(11)]{ps:weak-hopf}.  Unlike the previous
examples, this is not a verbatim translation of their proof; theirs
builds on previous lemmas, whereas the below proof is direct.
\begingroup
\allowdisplaybreaks
\begin{alignat*}{2}
  (t(x\s1), x\s2 y)
  &= (e\s1, x\s2 y \mid \cancel{x\s1 e\s2})\\
  &= (e\s1, (xe')\s2 y \mid\cancel{(xe')\s1 e\s2})\\
  &= (e\s1, x\s2 e'\s2 y \mid \cancel{x\s1 e'\s1 e\s2})\\
  &= (e\s1, x\s2 e\s3 y \mid \cancel{x\s1 e\s2})\\
  &= (e\s1, x\s2 e\ss22 y \mid \cancel{x\s1 e\ss21})\\
  &= (e\s1, (x e\s2)\s2 y \mid \cancel{(x e\s2)\s1})\\
  &= (e\s1, xe\s2 y)\\
  &= (e\s1, x (e\s2 y)\s2 \mid \cancel{(e\s2 y)\s1})\\
  &= (e\s1, x e\ss22 y\s2 \mid \cancel{e\ss21 y\s1})\\
  &= (e\s1, x e\s3 y\s2 \mid \cancel{e\s2 y\s1})\\
  &= (e\s1, x e'\s2 y\s2 \mid \cancel{e\s2e'\s1 y\s1})\\
  &= (e\s1, x (e'y)\s2 \mid \cancel{e\s2 (e'y)\s1})\\
  &= (e\s1, x y\s2 \mid \cancel{e\s2 y\s1})\\
  &= (s(y\s1), x y\s2).
\end{alignat*}
\endgroup

If $H$ and $K$ are weak bimonoids, a \textbf{weak bimonoid
  homomorphism} $f:H\to K$ must commute with both monoid and comonoid
structures as in \cref{sec:antipodes}.  It follows that it also
commutes with $s,t,r$; this is shown in~\cite[Lemma
1.2]{ps:weak-hopf}, and rendered below for $s$ in type theory:
\begin{align*}
  f(s(x))
  &= (f(e\s1)\mid \cancel{e\s2 x})\\
  &= (f(e\s1)\mid \cancel{f(e\s2 x)})\\
  &= (f(e\s1)\mid \cancel{f(e\s2)\, f(x)})\\
  &= (f(e)\s1\mid \cancel{f(e)\s2\, f(x)})\\
  &= (e\s1\mid \cancel{e\s2\, f(x)})\\
  &= s(f(x)).
\end{align*}

A weak bimonoid is a \textbf{weak Hopf monoid} if it has an
\textbf{antipode}: a morphism $x:H \types \inv x : H$ satisfying
\begin{align*}
  \inv{x\s1}\, x\s2 &= t(x)\\
  x\s1 \, \inv{x\s2} &= r(x)\\
  \inv{x\s1}\, x\s2\, \inv{x\s3} &= \inv{x}.
\end{align*}
This implies immediately that
\begin{mathpar}
\inv{x} = \inv{x\s1}\, x\s2\, \inv{x\s3} = \inv{x\s1} \; r(x\s2)\and
\inv{x} = \inv{x\s1}\, x\s2\, \inv{x\s3} = t(x\s1) \; \inv{x\s2}.
\end{mathpar}
As for ordinary bimonoids in \cref{sec:antipodes}, antipodes for weak
bimonoids are unique: if $x:H \types \inv x : H$ and
$x:H \types \invt x : H$ are both antipodes, we have:
\begin{align*}
  \inv{x}
  &= \inv{x\s1} \; r(x\s2)\\
  &= \inv{x\s1} \, x\ss21 \, \invt{x\ss22}\\
  &= \inv{x\ss11} \, x\ss12 \, \invt{x\s2}\\
  &= t(x\s1) \, \invt{x\s2}\\
  &= \invt{x}.
\end{align*}
Similarly, we can translate the argument of~\cite[Proposition
2.2]{ps:weak-hopf} that homomorphisms $f:H\to K$ of weak bimonoids
preserve antipodes: \begingroup \allowdisplaybreaks
\begin{align*}
  f(\inv{x})
  &= f(\inv{x\s1} \; r(x\s2)) \\
  &= f(\inv{x\s1})\, f(r(x\s2)) \\
  &= f(\inv{x\s1})\, r(f(x\s2)) \\
  &= f(\inv{x\s1})\, f(x\s2)\s1 \, \inv{f(x\s2)\s2} \\
  &= f(\inv{x\s1})\, f(x\ss22) \, \inv{f(x\ss22)} \\
  &= f(\inv{x\ss11})\, f(x\ss12) \, \inv{f(x\s2)} \\
  &= f(\inv{x\ss11}\; x\ss12) \, \inv{f(x\s2)} \\
  &= f(t(x\s1)) \, \inv{f(x\s2)} \\
  &= t(f(x\s1)) \, \inv{f(x\s2)} \\
  &= t(f(x)\s1) \, \inv{f(x)\s2} \\
  &= \inv{f(x)}.
\end{align*}
\endgroup

We end by sketching another approach to these sorts of uniqueness
results that is inspired by the ``types are like sets'' perspective of
our type theory.  For monoids in the category of sets, the uniqueness
of inverses is a \emph{pointwise} property: for any element $x$, if
$y$ and $z$ are two elements that are both inverses of $x$, then
$y=z$.  However, the uniqueness of antipodes as we have proven it
above, for both bimonoids and weak bimonoids, is instead a statement
about \emph{inversion operators} that apply to all elements at once.

The pointwise statement is stronger, and this makes for simpler
proofs.  For instance, we can simply show that a monoid homomorphism
$f:H\to K$ preserves inverses \emph{of elements} in the sense that if
$y$ is an inverse of $x$ then $f(y)$ is an inverse of $f(x)$, and
conclude directly from pointwise uniqueness of inverses in $K$ that
$f$ of ``the'' inverse of $x$ coincides with ``the'' inverse of
$f(x)$.

We can also formulate a ``pointwise'' sort of uniqueness of inverses
in the general case.  Suppose $H$ is a weak Hopf monoid, with given
antipode $x:H \types \inv{x}:H$.  Suppose also we have a comonoid $X$
and a comonoid morphism $x:X \types g(x) : H$, and also a morphism
$x:X \types i(x) : H$ such that
\begin{equation}
\begin{array}{rcl}
  i(x\s1)\, g(x\s2) &=& t(g(x))\\
  g(x\s1) \, i(x\s2) &=& r(g(x))\\
  i(x\s1) \, g(x\s2) \, i(x\s3) &=& i(x).
\end{array}\label{eq:rel-antipode}
\end{equation}
We think of $g$ as an $X$-indexed family of elements of $H$, and $i$
as a similarly indexed family of ``inverses''.  Then we can calculate
\begin{mathpar}
  t(g(x\s1))\, i(x\s2)
  = i(x\s1)\, g(x\s2)\, i(x\s3)
  = i(x)
  \and
  i(x\s1)\, r(g(x\s2))
  = i(x\s1) \, g(x\s2) \, i(x\s3)
  = i(x)
\end{mathpar}
and hence
\begin{align*}
  i(x) &= i(x\s1) \, r(g(x\s2))\\
  &= i(x\s1) \, g(x\s2) \, \inv{g(x\s3)}\\
  &= t(g(x\s1)) \, \inv{g(x\s2)}\\
  &= t(g(x)\s1) \, \inv{g(x)\s2}\\
  &= \inv{g(x)}.
\end{align*}
Note that in fact it suffices for $X$ to be a co-semigroup and $g$ to
be a co-semigroup morphism.  Moreover, instead of a single
co-semigroup, $X$ could be a context of co-semigroups.

Applying this with $X\defeq H$ and $g(x) \defeq x$, we obtain
uniqueness of the antipode itself.  But we can also derive
preservation of the antipode by a weak bimonoid homomorphism
$f:H\to K$, by taking $X\defeq H$ and $g\defeq f$ with
$i(x)\defeq f(\inv{x})$.  We simply check that this satisfies the
three properties~\eqref{eq:rel-antipode}:
\begin{align*}
  f(\inv{x\s1}) \, f(x\s2)
  &= f(\inv{x\s1}\, x\s2)\\
  &= f(t(x))\\
  &= t(f(x)).\\
  f(x\s1)\, f(\inv{x\s2})
  &= f(x\s1 \, \inv{x\s2})\\
  &= f(r(x))\\
  &= r(f(x)).\\
  f(\inv{x\s1}) \, f(x\s2) \, f(\inv{x\s3})
  &= f(\inv{x\s1} \, x\s2 \, \inv{x\s3})\\
  &= f(\inv{x}).
\end{align*}
Then the above argument shows immediately that
$f(\inv{x}) = \inv{f(x)}$, as desired.

As a final example, we show that for a weak Hopf monoid $H$, the
antipode is a weak monoid anti-homomorphism.  (Since Hopf monoids are
self-dual, this implies that the antipode is also a comonoid
anti-homomorphism, which was proven by a long string diagram
calculation in~\cite[Proposition 2.3]{ps:weak-hopf}).  Consider how we
prove the analogous statement in the category of sets, that inversion
in a group is a monoid anti-homomorphism.  Arguably the most natural
way is to simply check that $y^{-1} x^{-1}$ is an inverse of $x y$:
\begin{gather*}
  (y^{-1} x^{-1})(x y) = y^{-1} (x^{-1} x) y = y^{-1} e y = y^{-1} y = e\\
  (x y)(y^{-1} x^{-1}) = x (y y^{-1}) x^{-1} = x e x^{-1} = x x^{-1} = e
\end{gather*}
and conclude by uniqueness of inverses that $y^{-1} x^{-1} = (x y)^{-1}$.

Using our notion of pointwise uniqueness, we can reproduce this inside
our type theory to apply to weak Hopf monoids.  We take
$X \defeq (H,H)$, with its induced comonoid structure
$x:H, y:H \types (x\s1, y\s1, x\s2, y\s2) : (H,H,H,H)$.  We define
$g(x,y) \defeq xy$, which is a co-semigroup morphism (though not a
comonoid morphism) by the first axiom of a weak bimonoid, and
$i(x,y) \defeq \inv{y}\, \inv{x}$, and check the three
properties~\eqref{eq:rel-antipode}:
\begin{alignat*}{2}
  i((x,y)\s1) \, g((x,y)\s2)
  &= \inv{y\s1}\, \inv{x\s1} \, x\s2 \, y\s2\\
  &= \inv{y\s1}\, t(x) \, y\s2\\
  &= \inv{(t(x)\, y)\s1}\, (t(x)\, y)\s2 &\quad (4)\\
  &= t(t(x)\, y)\\
  &= t(xy) &\quad (3)\\
  &= t(g(x,y)).
\end{alignat*}
Here the line labeled (4) uses the so-numbered equation
$(y\s1,t(x)\, y\s2) = ((t(x)\,y)\s1,(t(x)\,y)\s2)$
from~\cite{ps:weak-hopf}, and similarly the line (3) is an instance of
their equation (3).  The next calculation is dual.
\begin{alignat*}{2}
  g((x,y)\s1) i((x,y)\s2)
  &= x\s1 \, y\s1 \, \inv{y\s1}\, \inv{x\s2}\\
  &= x\s1 \,r(y)\, \inv{x\s2}\\
  &= (x\,r(y))\s1\, \inv{(x\,r(y))\s2}\\
  &= r(x\,r(y))\\
  &= r(xy)\\
  &= r(g(x,y)).
\end{alignat*}
The final calculation uses the commutativity of $t$ with $r$, equation
(13) from~\cite{ps:weak-hopf} which we also proved above.
\begin{alignat*}{2}
  i((x,y)\s1) \, g((x,y)\s2) \, i((x,y)\s3)
  &= \inv{y\s1}\, \inv{x\s1} \, x\s2 \, y\s2 \, \inv{y\s3}\, \inv{x\s3}\\
  &= \inv{y\s1} \, t(x\s1) \, r(y\s2) \, \inv{x\s2}\\
  &= \inv{y\s1} \, r(y\s2)\, t(x\s1) \, \inv{x\s2}&\quad (13)\\
  &= \inv{y} \, \inv{x}\\
  &= i(x,y).
\end{alignat*}
We leave further applications to the interested reader.


\begin{references*}

\bibitem[Bar79]{barr:staraut}
Michael Barr.
\newblock {\em $\ast$-autonomous categories}, volume 752 of {\em Lecture Notes
  in Mathematics}.
\newblock Springer, 1979.

\bibitem[Bar91]{barr:staraut-ll}
Michael Barr.
\newblock *-autonomous categories and linear logic.
\newblock {\em Mathematical Structures in Computer Science}, 1(2):159–178,
  1991.

\bibitem[BCR18]{bcr:props}
John~C. Baez, Brandon Coya, and Franciscus Rebro.
\newblock Props in network theory.
\newblock {\em Theory and Applications of Categories}, 33(25):727--783, 2018.
\newblock arXiv:1707.08321.

\bibitem[BCST96]{bcst:natded-coh-wkdistrib}
R.~Blute, R.~Cockett, R.~Seely, and T.~Trimble.
\newblock Natural deduction and coherence for weakly distributive categories.
\newblock {\em J. Pure and Appl. Algebra}, 113:229--296, 1996.

\bibitem[Cam17]{campion:idem-selfdual}
Tim Campion.
\newblock If a $\otimes$-idempotent object has a dual, must it be self-dual?
\newblock MathOverflow, 2017.
\newblock \url{https://mathoverflow.net/q/277991} (version: 2017-08-09).

\bibitem[CPT16]{cpt:ll-session}
Lu{\'i}s Caires, Frank Pfenning, and Bernardo Toninho.
\newblock Linear logic propositions as session types.
\newblock {\em Mathematical Structures in Computer Science}, 26(3):367–423,
  2016.

\bibitem[CS97]{cs:wkdistrib}
Robin Cockett and Robert Seely.
\newblock Weakly distributive categories.
\newblock {\em Journal of Pure and Applied Algebra}, 114(2):133--173, 1997.
\newblock Corrected version available at
  \url{https://www.math.mcgill.ca/rags/linear/wdc-fix.pdf}.

\bibitem[Dun06]{duncan:types-qcomp}
Ross Duncan.
\newblock {\em Types for quantum computing}.
\newblock PhD thesis, Oxford University Computing Laboratory, 2006.

\bibitem[EL89]{el:draw-digraph}
P.~Eades and Xuemin Lin.
\newblock How to draw a directed graph.
\newblock {\em [Proceedings] 1989 IEEE Workshop on Visual Languages}, pages
  13--17, 1989.

\bibitem[Has97]{hasegawa:thesis}
Masahito Hasegawa.
\newblock {\em Models of Sharing Graphs: A Categorical Semantics of let and
  letrec}.
\newblock PhD thesis, University of Edinburgh, 1997.
\newblock
  \url{http://www.kurims.kyoto-u.ac.jp/~hassei/papers/ECS-LFCS-97-360.pdf}.

\bibitem[Joh02]{ptj:elephant}
Peter~T. Johnstone.
\newblock {\em Sketches of an Elephant: A Topos Theory Compendium: Volumes 1
  and 2}.
\newblock Number~43 in Oxford Logic Guides. Oxford Science Publications, 2002.

\bibitem[JS91]{js:geom-tenscalc-i}
Andr{\'e} Joyal and Ross Street.
\newblock The geometry of tensor calculus. {I}.
\newblock {\em Adv. Math.}, 88(1):55--112, 1991.

\bibitem[JSV96]{jsv:traced-moncat}
Andr{\'e} Joyal, Ross Street, and Dominic Verity.
\newblock Traced monoidal categories.
\newblock {\em Math. Proc. Cambridge Philos. Soc.}, 119(3):447--468, 1996.

\bibitem[Kis14]{kissinger:mat-hycat}
Aleks Kissinger.
\newblock Finite matrices are complete for (dagger-)hypergraph categories.
\newblock arXiv:1406.5942, 2014.

\bibitem[KL80]{kl:cpt}
G.~M. Kelly and M.~L. Laplaza.
\newblock Coherence for compact closed categories.
\newblock {\em J. Pure Appl. Algebra}, 19:193--213, 1980.

\bibitem[Mac65]{maclane:cat-alg}
Saunders MacLane.
\newblock Categorical algebra.
\newblock {\em Bull. Amer. Math. Soc.}, 71:40--106, 1965.

\bibitem[ML98]{maclane}
Saunders Mac~Lane.
\newblock {\em Categories For the Working Mathematician}, volume~5 of {\em
  Graduate Texts in Mathematics}.
\newblock Springer, second edition, 1998.

\bibitem[Ong96]{ong:sem-classical}
C.-H.~L. Ong.
\newblock A semantic view of classical proofs -- type-theoretic, categorical,
  and denotational characterizations (extended abstract).
\newblock In {\em Proceedings of LICS '96}, pages 230--241. IEEE Press, 1996.

\bibitem[PLC17]{co:flang-cyclic}
Jovana~Obradovi\'{c} Pierre-Louis~Curien.
\newblock A formal language for cyclic operads.
\newblock {\em Higher Structures}, 1(1), 2017.
\newblock arXiv:1602.07502.

\bibitem[PR84]{pr:spinors}
Roger Penrose and Wolfgang Rindler.
\newblock {\em Spinors and space-time. {V}ol. 1}.
\newblock Cambridge Monographs on Mathematical Physics. Cambridge University
  Press, Cambridge, 1984.
\newblock Two-spinor calculus and relativistic fields.

\bibitem[PS09]{ps:weak-hopf}
Craig Pastro and Ross Street.
\newblock Weak {Hopf} monoids in braided monoidal categories.
\newblock {\em Algebra Number Theory}, 3(2):149--207, 2009.

\bibitem[PS14]{ps:symtraces}
Kate Ponto and Michael Shulman.
\newblock Traces in symmetric monoidal categories.
\newblock {\em Expositiones Mathematicae}, 32(2):248--273, 2014.
\newblock arXiv:1107.6032.

\bibitem[Red93]{reddy:dirprolog}
Uday~S. Reddy.
\newblock A typed foundation for directional logic programming.
\newblock In E.~Lamma and P.~Mello, editors, {\em Extensions of Logic
  Programming}, pages 282--318, Berlin, Heidelberg, 1993. Springer Berlin
  Heidelberg.

\bibitem[Sel11]{selinger:graphical}
Peter Selinger.
\newblock A survey of graphical languages for monoidal categories.
\newblock In Bob Coeke, editor, {\em New Structures for Physics}, chapter~4.
  Springer, 2011.
\newblock arXiv:0908.3347.

\bibitem[Shi99]{shirahata:seqcalc-cptclosed}
Masaru Shirahata.
\newblock A sequent calculus for compact closed categories, 1999.
\newblock \url{http://www.fbc.keio.ac.jp/~sirahata/Research/cmll.ps}.

\bibitem[Sza75]{szabo:polycats}
M.E. Szabo.
\newblock Polycategories.
\newblock {\em Communications in Algebra}, 3(8):663--689, 1975.

\end{references*}

\end{document}